\documentclass[final]{amsart}

\usepackage[foot]{amsaddr}
\usepackage{booktabs}
\usepackage{fullpage}
\usepackage{tristan}
\usepackage{algorithm}
\usepackage{algpseudocode}

\usepackage{pgfplots}
\usepackage{pgfplotstable}
\usepackage{subcaption}
\usepackage{tikz}
\usepackage{todonotes}
\usepackage[giveninits=true,eprint=false,url=false,style=alphabetic]{biblatex}
\AtBeginDocument{%
  \setlength{\abovedisplayskip}{4pt plus 2pt minus 1pt}%
  \setlength{\belowdisplayskip}{4pt plus 2pt minus 1pt}%
  \setlength{\abovedisplayshortskip}{3pt plus 1pt minus 1pt}%
  \setlength{\belowdisplayshortskip}{3pt plus 1pt minus 1pt}%
  \setlength{\jot}{1.5pt}%
}
\AtBeginBibliography{\small\setlength{\bibitemsep}{0pt}}
\makeatletter
\AtBeginDocument{%
  \thm@preskip=.35\baselineskip\@plus.15\baselineskip\@minus.15\baselineskip
  \thm@postskip=.35\baselineskip\@plus.15\baselineskip\@minus.15\baselineskip
}
\makeatother

\newcommand{\Nh}{\mathcal N_h}
\newcommand{\Nhi}{\mathcal N_h^\circ}
\newcommand{\Vh}{V_h}
\newcommand{\Vht}{\widetilde V_h}
\newcommand{\muh}{\mu_h}

\newcommand{\dx}{\,\mathrm dx}

\newcommand{\uh}{u_h}

\title{A Finite Element Method for the Grad--Mercier Equation}

\author{Ben Ashby and Tristan Pryer} \date{}

\begin{document}

\maketitle

\begin{abstract}
We develop and analyse a conforming piecewise affine finite element
method for the Grad--Mercier equation, a nonlocal elliptic model for
magnetic equilibria whose source depends on the global rank of a flux
function. The principal difficulty is that the unique continuum
equilibrium has a maximum plateau of prescribed positive area. The
solution therefore samples the rank law only above the plateau mass,
while perturbations that separate equal values may distribute this mass
over several numerical ranks and thereby alter the non-strict rank
source.

We exploit the continuum well-posedness structure to replace the rank
law by an equivalent clipped, nonnegative law, and replace Lebesgue
measure by a mass-lumped nodal measure. Equal nodal values are assigned
a common rank interval, thereby retaining the full lumped mass of each
tied level.

The resulting weighted rearrangement gives a strictly convex discrete
energy and hence a unique finite element equilibrium. On meshes
satisfying a discrete comparison principle, the discrete equilibrium is
nonnegative and bounded above by a computable barrier. A
Poisson-preconditioned rearrangement iteration converges on each fixed
mesh and admits computable a posteriori energy and energy-norm bounds.
As polygonal finite element domains converge to the physical
cross-section, the discrete equilibria converge strongly in
$\hoz(\Omega)$ and uniformly on $\overline\Omega$ to the unique
Grad--Mercier equilibrium. Reconstructed equilibrium sources converge
weak-* in $\leb{\infty}(\Omega)$ to the original nonlocal source.
\end{abstract}

\section{Introduction}
\label{sec:introduction}

Tokamaks are toroidal devices designed to confine a hot fusion plasma by
means of magnetic fields. In a poloidal cross-section, a magnetic
equilibrium is described by a flux function whose level sets form nested
magnetic surfaces surrounding the magnetic axis. Plasma quantities are
naturally associated with these surfaces rather than with individual
spatial points. Consequently, constraints expressed in terms of the
distribution of magnetic flux lead to nonlocal equations involving
superlevel-set measures and monotone rearrangements; see
\cite{Temam1975,Temam1977,Temam1979,CaffarelliFarahRestrepo2025}.

The Grad--Mercier equation is one such model. It describes the poloidal
flux in a tokamak cross-section when the source is determined by the
position of each flux surface within the globally ordered family of flux
surfaces. We consider
\begin{equation}
  -\Delta u(x)
  =
  g\qp{
    \abs{u\geq u(x)}
  }
  \qquad
  \text{in }\Omega,
  \qquad
  u=0
  \qquad
  \text{on }\partial\Omega,
  \label{eq:intro-gm}
\end{equation}
where $\abs{u\geq t}$ denotes the measure of the superlevel set
$\qc{y\in\Omega:u(y)\geq t}$. Although the principal part is local, the
source at $x$ depends on the global ordering of the solution values through
the measure of the complete superlevel set determined by $u(x)$.

The analytical structure of \eqref{eq:intro-gm} was recently clarified by
Caffarelli, Farah and Restrepo
\cite{CaffarelliFarahRestrepo2025}. For a Lipschitz continuous,
nondecreasing source law $g$, the problem admits a unique solution with
$\cont{1,1}$ regularity. If $g$ has a unique zero
$\alpha\in(0,\abs{\Omega})$, then the solution develops a \emph{dead core}
\begin{equation}
  D_u
  :=
  \qc{
    x\in\Omega:
    u(x)=\max_{\overline\Omega}u
  },
  \qquad
  \abs{D_u}
  =
  \alpha.
  \label{eq:intro-dead-core}
\end{equation}
The source vanishes on $D_u$ and is positive away from it, so the problem
may be viewed as a nonlocal, volume-selected free-boundary problem.

The positive-measure maximum plateau is also the principal obstacle to
numerical approximation. If $M=\max_{\Omega}u,$ then
\begin{equation}
  \abs{u>M}=0,
  \qquad
  \abs{u\geq M}=\alpha.
  \label{eq:maximum-rank-gap}
\end{equation}
Thus the plateau carries the full mass $\alpha$, even though no point
lies strictly above the maximum value. A perturbation that separates
equal values may distribute this mass over several numerical ranks and
thereby alter the source assigned near the maximum. 

The Grad--Mercier problem belongs to the classical theory of
rearrangement-dependent plasma equations. Early existence and qualitative
results were developed by Temam, Mossino, Gourgeon and their collaborators
\cite{Temam1975,Temam1977,Temam1979,Mossino1979,
GourgeonMossino1979,MossinoTemam1981}. Relative rearrangements were later
introduced for more general nonlocal plasma models
\cite{DiazRakotoson1996,Rakotoson2008}. Finite element and Galerkin
approximations of related rearrangement problems were studied in
\cite{BermudezSeoane1998,Rakotoson1999,RakotosonSeoane2000}, while local
plasma free-boundary, nonlinear eigenvalue and volume-matching problems
were considered in
\cite{BarrettElliott1989,Caloz1991,BarrettElliott1991,
BarrettChakrabarti1991}.

The existing rearrangement approximations do not resolve the difficulty
created by the prescribed Grad--Mercier plateau. In particular, key
consistency results in \cite{RakotosonSeoane2000} are formulated under
the absence of plateaux, whereas the maximum plateau in
(\ref{eq:intro-gm}) is intrinsic to the continuum solution
and has positive area. The present method instead assigns every tied
nodal level a common rank interval, thereby retaining the full lumped
mass of that level and treating the non-strict rank source directly.

The main contributions are fourfold. First, we introduce a
mass-weighted nodal rearrangement in which tied nodal values occupy a
common rank interval. This gives a strictly convex discrete energy and
hence a unique finite element solution, with positivity and an upper
barrier on meshes satisfying a discrete comparison principle. Second,
we characterise the discrete supergradients by rank-allocation
sources, prove convergence of a Poisson-preconditioned rearrangement
iteration on each fixed mesh, and derive computable a posteriori
energy and energy-norm bounds. Third, we prove strong convergence of
the discrete minimisers in $\hoz(\Omega)$ and $\leb{\infty}(\Omega)$
as the polygonal domains and meshes approximate the physical
domain. Finally, we identify the weak-star limit of the reconstructed
equilibrium sources. The limiting function is therefore the unique
continuum Grad--Mercier equilibrium, whose maximum set has measure
$\alpha$ and whose remaining level sets have zero measure.

The proof combines consistency of the weighted nodal rearrangement with
compactness for mass-lumped Poisson problems on varying polygonal
domains and the variational characterisation of the continuum
solution. Unlike earlier consistency arguments requiring atom-free
distributions, it accommodates the positive-measure maximum level.

Section~\ref{sec:problem} recalls the continuum problem,
Section~\ref{sec:numerical-method} defines the discrete method and
states the main results, Section~\ref{sec:numerical-experiments}
presents numerical experiments, and Section~\ref{sec:proofs} contains
the analysis.

\section{Problem setup}
\label{sec:problem}

Let $\Omega\subset\reals^2$ be a bounded, connected, convex domain
with $\cont{1,1}$ boundary. Throughout, $\abs{\cdot}$ denotes
two-dimensional Lebesgue measure. Following
\cite{CaffarelliFarahRestrepo2025}, for a measurable function
$v:\Omega\to\reals$ and $t\in\reals$, we write
\begin{equation}
  v\geq t
  :=
  \qc{
    x\in\Omega:
    v(x)\geq t
  },
  \qquad
  v>t
  :=
  \qc{
    x\in\Omega:
    v(x)>t
  }.
  \label{eq:superlevel-notation}
\end{equation}
The measures of these sets are denoted by $\abs{v\geq t}$ and
$\abs{v>t}$.

For later use, we also recall the decreasing rearrangement associated
with this superlevel-set convention. If $v\in L^1(\Omega)$, its
decreasing rearrangement
$v^\ast:(0,\abs{\Omega})\longrightarrow\mathbb R$ is the generalised
inverse of its distribution function:
\begin{equation}
  v^\ast(s)
  :=
  \inf\qc{
    t\in\mathbb R:
    \abs{v>t}\leq s
  },
  \qquad
  0<s<\abs{\Omega}.
  \label{eq:continuum-decreasing-rearrangement}
\end{equation}
Thus $v^\ast$ is nonincreasing and is equimeasurable with $v$.

We consider the Grad--Mercier problem
\begin{equation}
  -\Delta u(x)
  =
  g\qp{
    \abs{u\geq u(x)}
  }
  \qquad\text{in }\Omega,
  \qquad
  u=0
  \qquad\text{on }\partial\Omega.
  \label{eq:grad-mercier}
\end{equation}

\begin{assumption}[Assumption on source law]
\label{ass:rank-law}
The function $ g:[0,\abs{\Omega}]\longrightarrow\reals$ is Lipschitz
continuous and nondecreasing. There exists a unique
$\alpha\in(0,\abs{\Omega})$ such that $g(\alpha)=0.$ Consequently,
\begin{equation}
  g(r)<0
  \quad\text{for }r<\alpha,
  \qquad
  g(r)>0
  \quad\text{for }r>\alpha.
  \label{eq:g-sign}
\end{equation}
\end{assumption}

\begin{definition}[Weak solution]
\label{def:weak-solution}
A function $u\in \hoz(\Omega)\cap \cont0(\overline\Omega)$ is a weak
solution of \eqref{eq:grad-mercier} if
\begin{equation}
  \int_\Omega
  \nabla u\cdot\nabla\phi\dx
  =
  \int_\Omega
  g\qp{
    \abs{u\geq u(x)}
  }
  \phi(x)\dx
  \Foreach \phi\in \hoz(\Omega).
  \label{eq:weak-problem}
\end{equation}
\end{definition}

\begin{theorem}[Continuum well-posedness and dead core]
\label{thm:continuum-well-posedness}
Let Assumption~\ref{ass:rank-law} hold. Then \eqref{eq:grad-mercier}
admits a unique weak solution $u$. Moreover, $u\in
\cont{1,1}(\overline\Omega), \Delta u\in \cont0(\overline\Omega).$ The
maximum set
\begin{equation}
  D_u
  :=
  \qc{
    x\in\Omega:
    u(x)=\max_{\overline\Omega}u
  }
  \label{eq:continuum-dead-core}
\end{equation}
has the prescribed measure
\begin{equation}
  \abs{D_u}=\alpha.
  \label{eq:continuum-dead-core-measure}
\end{equation}
In addition,
\begin{equation}
  \abs{u\geq u(x)}
  =
  \alpha
  \qquad
  \text{for }x\in D_u,
  \label{eq:superlevel-measure-on-dead-core}
\end{equation}
whereas
\begin{equation}
  \abs{u\geq u(x)}
  >
  \alpha
  \qquad
  \text{for }x\in\Omega\setminus D_u.
  \label{eq:superlevel-measure-off-dead-core}
\end{equation}
\end{theorem}

\begin{proof}
Existence, uniqueness, regularity, and
\eqref{eq:continuum-dead-core-measure} follow from
\cite{CaffarelliFarahRestrepo2025}. For $x\in D_u$, the superlevel set
at $u(x)$ is $D_u$, giving
\eqref{eq:superlevel-measure-on-dead-core}. If $x\notin D_u$, then
$u(x)<\max_{\overline\Omega}u$; by continuity,
$\{u>u(x)\}$ contains a neighbourhood of $D_u$ and hence has measure
strictly greater than $\alpha$. This proves
\eqref{eq:superlevel-measure-off-dead-core}.
\end{proof}

Theorem~\ref{thm:continuum-well-posedness} shows that the continuum
solution samples $g$ only on $[\alpha,\abs{\Omega}]$. We therefore define
the clipped rank law
\begin{equation}
  G(\rho)
  :=
  g\qp{\max\qc{\alpha,\rho}},
  \qquad
  M_\alpha
  :=
  g\qp{\abs{\Omega}}.
  \label{eq:clipped-rank-law}
\end{equation}
Then $G$ is Lipschitz continuous and nondecreasing, and $0\leq
G(\rho)\leq M_\alpha$, $G(\rho)=0$ for $0\leq\rho\leq\alpha$.
Moreover, $G\qp{\abs{u\geq u(x)}} = g\qp{\abs{u\geq u(x)}}.$ Thus, as
illustrated in Figure \ref{fig:continuum-plateau-mechanism}, $G$
extends the continuum source to a nonnegative law on the full rank
interval.

\begin{figure}[tbp]
  \centering
  \begin{subfigure}[t]{0.30\textwidth}
    \centering
    \includegraphics[width=\linewidth]
      {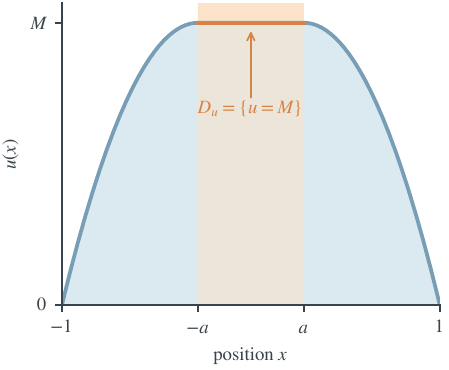}
    \caption{A positive-measure maximum plateau.}
    \label{fig:continuum-plateau-profile}
  \end{subfigure}
  \hfill
  \begin{subfigure}[t]{0.30\textwidth}
    \centering
    \includegraphics[width=\linewidth]
      {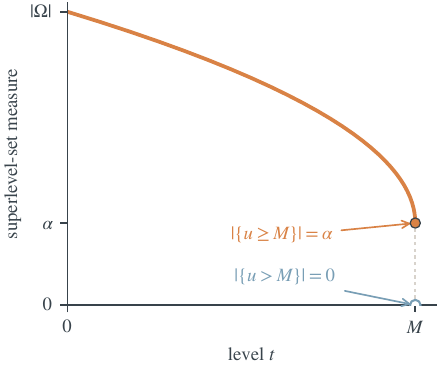}
    \caption{The rank jump at the maximum.}
    \label{fig:continuum-rank-maximum}
  \end{subfigure}
  \hfill
  \begin{subfigure}[t]{0.30\textwidth}
    \centering
    \includegraphics[width=\linewidth]
      {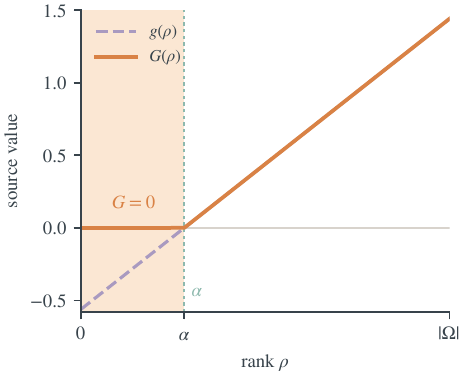}
    \caption{The original and clipped laws.}
    \label{fig:continuum-clipped-laws}
  \end{subfigure}
  \caption{ The continuum plateau mechanism. In the one-dimensional
    schematic, the plateau has measure $2a=\alpha$.  }
  \label{fig:continuum-plateau-mechanism}
\end{figure}

\section{The numerical method}
\label{sec:numerical-method}

For $\delta>0$, let $\Omega_{-\delta} := \qc{x\in\Omega: \dist
  (x,\partial\Omega)>\delta}.$ Let $\Omega_h$ be convex polygonal
domains satisfying $\Omega_{-\delta_h}\subset\Omega_h\subset\Omega,$
for $\delta_h\to0.$ Then $\abs{\Omega\setminus\Omega_h}\to0$ and
$\abs{\Omega_h}\to\abs{\Omega}$; throughout we assume $h$ is
sufficiently small that $\alpha<\abs{\Omega_h}$.

Let $\Th$ be a conforming triangulation of $\overline{\Omega_h}$ and
set $h := \max_{K\in\Th} \operatorname{diam}(K).$ The family
$\{\Th\}_{h\downarrow0}$ is assumed uniformly shape-regular and
quasi-uniform. We write $\Nh$ and $\Nhi$ for the sets of all vertices
and interior vertices, respectively, and define
\begin{equation}
  \Vht
  :=
  \qc{
    v_h\in \cont0(\overline{\Omega_h}):
    v_h|_K\in\mathbb P_1(K)
    \text{ for every }K\in\Th
  },
  \qquad
  \Vh
  :=
  \Vht\cap \hoz(\Omega_h).
\end{equation}
Let $\{\varphi_i\}_{i\in\Nh}$ be the nodal basis associated with the
vertices $\{x_i\}_{i\in\Nh}$. Then 
\begin{equation}
  \varphi_i\geq0,
  \qquad
  \sum_{i\in\Nh}\varphi_i=1
  \qquad
  \text{in }\Omega_h.
  \label{eq:partition-unity}
\end{equation}

For $v_h\in\Vh$, let $\mathcal E_hv_h$ denote its extension by zero to
$\Omega$. Then
\begin{equation}
  \mathcal E_hv_h\in \hoz(\Omega),
  \qquad
  \Norm{\mathcal E_hv_h}_{L^2(\Omega)}
  =
  \Norm{v_h}_{L^2(\Omega_h)},
  \qquad
  \Norm{\nabla\mathcal E_hv_h}_{L^2(\Omega)}
  =
  \Norm{\nabla v_h}_{L^2(\Omega_h)}.
  \label{eq:zero-extension-norms}
\end{equation}

Define the lumped nodal weights and the discrete measure by
\begin{equation}
  \omega_i
  :=
  \int_{\Omega_h}\varphi_i\dx,
  \qquad
  \muh
  :=
  \sum_{i\in\Nh}
  \omega_i\delta_{x_i}.
  \label{eq:nodal-measure}
\end{equation}
They satisfy
\begin{equation}
  \omega_i>0,
  \qquad
  \sum_{i\in\Nh}\omega_i
  =
  \abs{\Omega_h},
  \qquad
  \max_{i\in\Nh}\omega_i
  \leq
  Ch^2
  \longrightarrow0.
  \label{eq:nodal-weight-properties}
\end{equation}

The finite element bilinear form is
\begin{equation}
  a_h(v,w)
  :=
  \int_{\Omega_h}
  \nabla v\cdot\nabla w\dx.
  \label{eq:bilinear-form}
\end{equation}
For $v_h,w_h\in\Vh$,
\begin{equation}
  a_h(v_h,w_h)
  =
  \int_\Omega
  \nabla\mathcal E_hv_h
  \cdot
  \nabla\mathcal E_hw_h
  \dx.
  \label{eq:extended-bilinear-form}
\end{equation}

\begin{assumption}[Discrete maximum principle]
\label{ass:DMP}
For every $h$, the interior stiffness matrix $A_h := \qp{
  a_h(\varphi_j,\varphi_i) }_{i,j\in\Nhi} $ is an irreducible
$M$-matrix.
\end{assumption}

For conforming Delaunay triangulations, the stiffness matrix has
nonpositive off-diagonal entries. Assumption~\ref{ass:DMP} follows when
the graph of strictly negative interior couplings is connected; strictly
Delaunay or suitably connected non-obtuse triangulations are sufficient.
The resulting comparison principle is
\begin{equation}
  a_h(v_h,\varphi_i)
  \leq
  a_h(w_h,\varphi_i)
  \quad
  \text{for every }i\in\Nhi
  \quad\Longrightarrow\quad
  v_h(x_i)\leq w_h(x_i)
  \quad
  \text{for every }i\in\Nhi.
  \label{eq:discrete-comparison-summary}
\end{equation}

\subsection{Discrete rearrangement and variational problem}

Each nodal value $v_h(x_i)$ carries the mass $\omega_i$. The weighted
nodal rearrangement orders these values from largest to smallest while
retaining their masses; see
\cite{MossinoTemam1981,Rakotoson1999,RakotosonSeoane2000}. See also
Figure \ref{fig:discrete-rearrangement-schematic} for geometric
illustrations of the concepts introduced in this Section.

Let $\mathcal L_h(v_h) := \qc{ v_h(x_i): i\in\Nh }$ be the set of
distinct nodal values. For $c\in\mathcal L_h(v_h)$, define the rank
interval occupied by the level $c$ by
\begin{equation}
  \mathcal I_h(c;v_h)
  :=
  \left(
    \sum_{\substack{i\in\Nh\\v_h(x_i)>c}}
    \omega_i,
    \,
    \sum_{\substack{i\in\Nh\\v_h(x_i)\geq c}}
    \omega_i
  \right].
  \label{eq:nodal-rank-interval}
\end{equation}
Its length is the total lumped mass carried by that level (see Figure
\ref{fig:discrete-rearrangement-schematic}):
\begin{equation}
  \abs{\mathcal I_h(c;v_h)}
  =
  \sum_{\substack{i\in\Nh\\v_h(x_i)=c}}
  \omega_i.
  \label{eq:nodal-level-mass}
\end{equation}

The weighted decreasing rearrangement of $v_h$ is the step function
$v_h^\ast:(0,\abs{\Omega_h}]\to\reals$ defined by $v_h^\ast(\rho) :=
  c$ when $\rho\in\mathcal I_h(c;v_h).$ The discrete rearrangement
  functional is therefore
\begin{equation}
  \begin{aligned}
  \mathcal R_h(v_h)
  &:=
  \int_0^{\abs{\Omega_h}}
  v_h^\ast(\rho)G(\rho)\mathop{}\d\rho
  \\
  &=
  \sum_{c\in\mathcal L_h(v_h)}
  c
  \int_{\mathcal I_h(c;v_h)}
  G(\rho)\mathop{}\d\rho.
  \end{aligned}
  \label{eq:discrete-rearrangement-functional}
\end{equation}
Since $G$ is nondecreasing, $\mathcal R_h$ is concave and continuous, motivating the definition 
\begin{equation}
  \mathcal J_h(v_h)
  :=
  \frac12
  a_h(v_h,v_h)
  -
  \mathcal R_h(v_h).
  \label{eq:discrete-energy}
\end{equation}

\begin{definition}[Discrete Grad--Mercier problem]
\label{def:discrete-problem}
The discrete problem is to find
\begin{equation}
  \uh
  :=
  \operatorname*{argmin}_{v_h\in\Vh}
  \mathcal J_h(v_h).
  \label{eq:discrete-problem}
\end{equation}
Equivalently, $\uh\in\Vh$ satisfies the variational inequality
\begin{equation}
  a_h(\uh,v_h-\uh)
  \geq
  \mathcal R_h(v_h)
  -
  \mathcal R_h(\uh)
  \qquad
  \text{for every }v_h\in\Vh.
  \label{eq:discrete-variational-inequality}
\end{equation}
\end{definition}

We next describe the supergradients appearing in the optimality
condition of (\ref{eq:discrete-problem}). For $v_h\in\Vh$ and
$i\in\Nh$, set
\begin{equation}
  \mathcal I_{h,i}[v_h]
  :=
  \mathcal I_h\qp{
    v_h(x_i);v_h
  }.
  \label{eq:nodal-rank-interval-at-node}
\end{equation}
Thus $\mathcal I_{h,i}[v_h]$ is the rank interval occupied by the
nodal level containing $v_h(x_i)$. The canonical nodal source
$\overline q_h[v_h] := \qp{ \overline q_{h,i}[v_h] }_{i\in\Nh}$ has
components
\begin{equation}
  \overline q_{h,i}[v_h]
  :=
  \frac{1}{
    \abs{\mathcal I_{h,i}[v_h]}
  }
  \int_{\mathcal I_{h,i}[v_h]}
  G(\rho)\mathop{}\d\rho.
  \label{eq:canonical-nodal-source}
\end{equation}
This is well defined because every nodal rank interval has strictly
positive mass.

In particular,
\begin{equation}
  \abs{
    \mathcal I_{h,i}[v_h]
  }
  =
  \omega_i
  \label{eq:untied-rank-interval-mass}
\end{equation}
when $v_h(x_i)$ is untied. Moreover,
\begin{equation}
  0
  \leq
  \overline q_{h,i}[v_h]
  \leq
  M_\alpha,
  \label{eq:canonical-source-bound}
\end{equation}
and
\begin{equation}
  \mathcal R_h(w_h)-\mathcal R_h(v_h)
  \leq
  \sum_{i\in\Nh}
  \omega_i
  \overline q_{h,i}[v_h]
  \qp{
    w_h(x_i)-v_h(x_i)
  }
  \label{eq:canonical-source-inequality}
\end{equation}
for every $w_h\in\Vh$.

More generally, a nodal vector $q_h=(q_{h,i})_{i\in\Nh}$ is an
admissible rearrangement source at $v_h$ if
\begin{equation}
  0\leq q_{h,i}\leq M_\alpha
  \qquad
  \text{for every }i\in\Nh
  \label{eq:admissible-source-bound}
\end{equation}
and
\begin{equation}
  \mathcal R_h(w_h)-\mathcal R_h(v_h)
  \leq
  \sum_{i\in\Nh}
  \omega_iq_{h,i}
  \qp{
    w_h(x_i)-v_h(x_i)
  }
  \label{eq:admissible-source-inequality}
\end{equation}
for every $w_h\in\Vh$. The canonical source is admissible, but it need
not be uniquely admissible on a tied level.

The optimality condition for \eqref{eq:discrete-problem} is equivalent
to the existence of an admissible source
$q_h^\ast=(q_{h,i}^\ast)_{i\in\Nh}$ at $\uh$ such that
\begin{equation}
  a_h(\uh,v_h)
  =
  \sum_{i\in\Nh}
  \omega_iq_{h,i}^\ast v_h(x_i)
  \qquad
  \text{for every }v_h\in\Vh.
  \label{eq:discrete-source-equation}
\end{equation}

For later use, define the continuous reconstruction
\begin{equation}
  s_h[q_h]
  :=
  \sum_{i\in\Nh}
  q_{h,i}\varphi_i
  \in\Vht.
  \label{eq:nodal-source-reconstruction}
\end{equation}
If $0\leq q_{h,i}\leq M_\alpha$, then
\begin{equation}
  0
  \leq
  s_h[q_h]
  \leq
  M_\alpha
  \qquad
  \text{in }\Omega_h.
  \label{eq:reconstructed-source-bound}
\end{equation}
Figure~\ref{fig:discrete-rearrangement-schematic} illustrates the
construction using the radial profile of
Section~\ref{sec:numerics-radial}. Here $M_h:=\max_{i\in\Nh}v_h(x_i),$
$D_h:=\qc{x\in\Omega_h:v_h(x)=M_h}.$
\begin{figure}[htbp]
  \centering
  \begin{subfigure}[t]{0.45\textwidth}
    \centering
    \includegraphics[width=\linewidth]
      {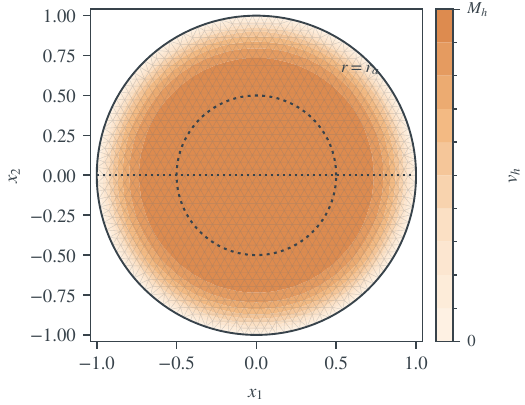}
    \caption{The finite element profile on $\Th$.}
    \label{fig:discrete-radial-profile}
  \end{subfigure}
  \hfill
  \begin{subfigure}[t]{0.45\textwidth}
    \centering
    \includegraphics[width=\linewidth]
      {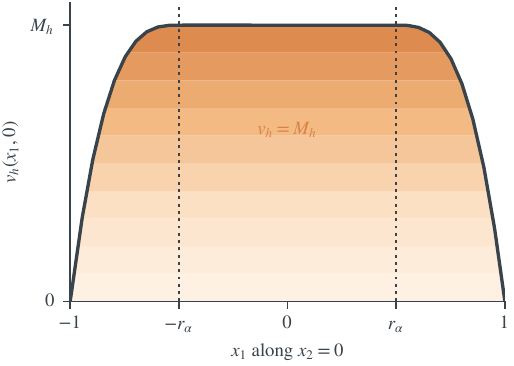}
    \caption{The profile along $x_2=0$.}
    \label{fig:discrete-diameter-profile}
  \end{subfigure}

  \begin{subfigure}[t]{0.45\textwidth}
    \centering
    \includegraphics[width=\linewidth]
      {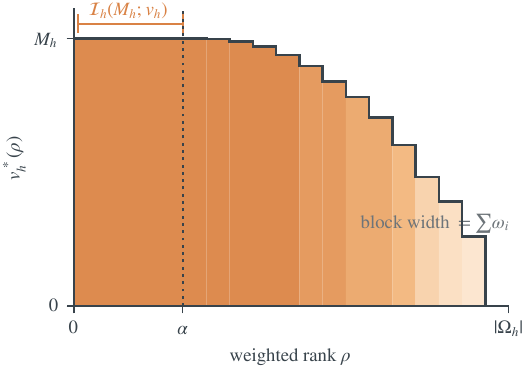}
    \caption{The weighted rearrangement $v_h^\ast$.}
    \label{fig:discrete-weighted-rearrangement}
  \end{subfigure}
  \hfill
  \begin{subfigure}[t]{0.45\textwidth}
    \centering
    \includegraphics[width=\linewidth]
      {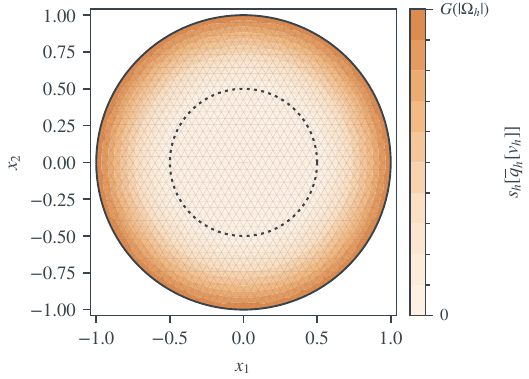}
    \caption{The source reconstruction
      $s_h[\overline q_h[v_h]]$.}
    \label{fig:discrete-canonical-source}
  \end{subfigure}
  \caption{
    The weighted nodal rearrangement and its canonical tied-level
    source. Equal nodal values occupy one rank interval
    $\mathcal I_h(c;v_h)$, whose width is their total lumped mass.
    The dotted lines mark $r=r_\alpha$ or the corresponding rank
    $\rho=\alpha$.
  }
  \label{fig:discrete-rearrangement-schematic}
\end{figure}

Finally, let $z_h\in\Vh$ be the discrete upper barrier solving
\begin{equation}
  a_h(z_h,v_h)
  =
  M_\alpha
  \int_{\Omega_h}v_h\dx
  \qquad
  \text{for every }v_h\in\Vh,
  \label{eq:discrete-barrier}
\end{equation}
and set
\begin{equation}
  K_h
  :=
  \qc{
    v_h\in\Vh:
    0\leq v_h(x_i)\leq z_h(x_i)
    \text{ for every }i\in\Nhi
  }.
  \label{eq:invariant-box}
\end{equation}
The source bounds and the discrete comparison principle imply
\begin{equation}
  \uh\in K_h.
  \label{eq:discrete-solution-bound}
\end{equation}

\subsection{Rearrangement iteration}

Let $u_h^0\in K_h$. Given $u_h^n$, set
\begin{equation}
  q_{h,i}^n
  :=
  \overline q_{h,i}[u_h^n],
  \qquad
  i\in\Nh,
  \label{eq:iterative-rank-source}
\end{equation}
and compute $\widehat u_h^{n+1}\in\Vh$ from
\begin{equation}
  a_h(\widehat u_h^{n+1},v_h)
  =
  \sum_{i\in\Nh}
  \omega_iq_{h,i}^n v_h(x_i)
  \qquad
  \text{for every }v_h\in\Vh.
  \label{eq:rearrangement-poisson-step}
\end{equation}
The relaxed update is
\begin{equation}
  u_h^{n+1}
  =
  (1-\theta_n)u_h^n
  +
  \theta_n\widehat u_h^{n+1},
  \label{eq:rearrangement-iteration}
\end{equation}
where
\begin{equation}
  0<\theta_n\leq1,
  \qquad
  \sum_{n=0}^\infty\theta_n=\infty,
  \qquad
  \sum_{n=0}^\infty\theta_n^2<\infty.
  \label{eq:step-size-conditions}
\end{equation}
For example,
\begin{equation}
  \theta_n
  =
  \frac{\theta_0}{n+1},
  \qquad
  0<\theta_0\leq1,
  \label{eq:example-step-size}
\end{equation}
satisfies these conditions.

This is the subgradient method for $\mathcal J_h$ preconditioned by the
discrete Poisson operator. The source bounds and comparison principle
give
\begin{equation}
  u_h^0\in K_h
  \quad\Longrightarrow\quad
  u_h^n\in K_h
  \qquad
  \text{for every }n\geq0.
  \label{eq:iteration-invariance}
\end{equation}
Each step requires one sorting operation, one source assembly, and one
Poisson solve with a fixed stiffness matrix. The sorting cost is
$O(N_h\log N_h)$.

\begin{theorem}[A posteriori energy bound]
\label{thm:discrete-energy-certificate}
Let $\qc{ u_h^j,\widehat u_h^{j+1} }_{j=0}^n$ be generated by the
rearrangement iteration
\eqref{eq:iterative-rank-source}--\eqref{eq:rearrangement-iteration}.
For $0\leq j\leq n$, define the stepwise object
\begin{equation}
  \ell_h^j
  :=
  a_h\qp{
    \widehat u_h^{j+1},u_h^j
  }
  -
  \mathcal R_h(u_h^j)
  -
  \frac12
  a_h\qp{
    \widehat u_h^{j+1},
    \widehat u_h^{j+1}
  }.
  \label{eq:iteration-lower-bound}
\end{equation}
The best floor available after $n$ iterations, the current iterate
energy, and their gap are
\begin{equation}
  L_h^n
  :=
  \max_{0\leq j\leq n}\ell_h^j,
  \qquad
  \mathcal J_h^n
  :=
  \mathcal J_h(u_h^n),
  \qquad
  \eta_h^n
  :=
  \mathcal J_h^n-L_h^n.
  \label{eq:iteration-energy-bounds}
\end{equation}
Writing $\mathcal J_h^\star := \mathcal J_h(u_h)$ for the minimum
energy, where $u_h$ is the unique discrete minimiser, we have
\begin{equation}
  L_h^n
  \leq
  \mathcal J_h^\star
  \leq
  \mathcal J_h^n.
  \label{eq:certified-energy-bracket}
\end{equation}
Consequently,
\begin{equation}
  0
  \leq
  \mathcal J_h^n-\mathcal J_h^\star
  \leq
  \eta_h^n,
  \label{eq:certified-energy-error}
\end{equation}
and
\begin{equation}
  \Norm{
    \nabla(u_h^n-u_h)
  }_{\leb2(\Omega_h)}
  \leq
  \sqrt{2\eta_h^n}.
  \label{eq:certified-energy-norm-error}
\end{equation}
\end{theorem}

\begin{theorem}[Well-posedness and convergence of the discrete method]
\label{thm:main-nonlinear}
Let the domain and mesh hypotheses above hold, and let
Assumptions~\ref{ass:rank-law} and~\ref{ass:DMP} hold.

For every sufficiently small $h$, the discrete problem
\eqref{eq:discrete-problem} has a unique solution $\uh\in K_h.$

For every $u_h^0\in K_h$ and every step sequence satisfying
\eqref{eq:step-size-conditions}, the iteration
\eqref{eq:rearrangement-poisson-step}--\eqref{eq:rearrangement-iteration}
converges to $\uh$:
\begin{equation}
  \Norm{
    u_h^n-\uh
  }_{\hoz(\Omega_h)}
  \longrightarrow0
  \qquad
  \text{as }n\longrightarrow\infty.
  \label{eq:iteration-convergence}
\end{equation}
For fixed $h$, convergence also holds in $L^\infty(\Omega_h)$.

Let $u$ be the continuum solution from
Theorem~\ref{thm:continuum-well-posedness}. Then, as $h\to0$,
\begin{align}
  \mathcal E_h\uh
  &\longrightarrow u
  &&\text{strongly in }\hoz(\Omega),
  \label{eq:main-H1-convergence}
  \\
  \mathcal E_h\uh
  &\longrightarrow u
  &&\text{strongly in }\leb{\infty}(\Omega).
  \label{eq:main-Linf-convergence}
\end{align}

For each $h$, choose any admissible equilibrium source
$q_h^\ast$ satisfying \eqref{eq:discrete-source-equation}, and define
\begin{equation}
  s_h
  :=
  s_h[q_h^\ast].
  \label{eq:equilibrium-source-reconstruction}
\end{equation}
Then, independently of the choice of $q_h^\ast$, its zero extension
satisfies
\begin{equation}
  \widetilde{s_h}
  \stackrel{*}{\rightharpoonup}
  g\qp{
    \abs{u\geq u(\cdot)}
  }
  \qquad
  \text{in }\leb{\infty}(\Omega)
  \quad\text{as }h\to0.
  \label{eq:main-source-convergence}
\end{equation}
Equivalently,
\begin{equation}
\int_\Omega
\widetilde{s_h}\varphi\dx
\longrightarrow
  \int_\Omega
  g\qp{
    \abs{u\geq u(\cdot)}
    }
  \varphi\dx
  \qquad
  \text{for every }\varphi\in \leb 1(\Omega).
  \label{eq:main-source-convergence-duality}
\end{equation}
In particular, the limiting solution has a maximum plateau of measure
$\alpha$ and no positive-measure level sets below its maximum.
\end{theorem}

\section{Numerical experiments}
\label{sec:numerical-experiments}

The numerical study has two parts. The first verifies the method against
an exact radial Grad--Mercier solution and measures convergence under mesh
refinement. The second demonstrates the method on a smooth, strictly
convex plasma cross-section with elongation and triangularity.

On each mesh, the discrete minimiser is approximated by the
rearrangement iteration
\eqref{eq:iterative-rank-source}--\eqref{eq:rearrangement-iteration},
starting from $u_h^0=0$. Two initial full steps are followed by
$\theta_n = \frac14 \qp{1+\frac{n}{32}}^{-3/5}.$

Each iteration requires one sorting operation, one canonical source
assembly, and one Poisson solve with the fixed stiffness matrix.

At every iteration, we record the current energy $\mathcal J_h^n :=
\mathcal J_h(u_h^n),$ $\eta_h^n := \mathcal J_h^n-L_h^n,$ where
$L_h^n$ is the lower bound from
Theorem~\ref{thm:discrete-energy-certificate}. We terminate the
iteration when
\begin{equation}
  \sqrt{2\eta_h^n}
  \leq
  10^{-4}.
  \label{eq:iteration-stopping-criterion}
\end{equation}
The reported iteration count is the number of Poisson solves required
to satisfy this criterion.

\subsection{Exact radial verification}
\label{sec:numerics-radial}

Let
\begin{equation}
  \Omega
  =
  B_1(0),
  \qquad
  g(\rho)
  =
  \rho-\alpha,
  \qquad
  \alpha
  =
  \frac{\pi}{4},
  \qquad
  G(\rho)
  =
  (\rho-\alpha)_+.
  \label{eq:radial-test-problem}
\end{equation}
The dead-core radius is
\begin{equation}
  r_\alpha
  =
  \sqrt{\frac{\alpha}{\pi}}
  =
  \frac12.
  \label{eq:radial-dead-core-radius}
\end{equation}
For a radial nonincreasing solution, $\abs{\qc{u\geq u(r)}} = \pi r^2$
for $r>r_\alpha,$ and hence
\begin{equation}
  s(r)
  =
  \pi
  \qp{
    r^2-r_\alpha^2
  }_+.
  \label{eq:radial-exact-source}
\end{equation}
The exact solution is
\begin{equation}
  u(r)
  =
  \begin{cases}
    \displaystyle
    u(r_\alpha),
      &0\leq r\leq r_\alpha,
      \\[2mm]
    \displaystyle
    \frac{\pi}{4}
    \qb{
      \frac{1-r^4}{4}
      -
      r_\alpha^2(1-r^2)
      +
      r_\alpha^4
      \log\qp{\frac1r}
    },
      &r_\alpha<r\leq1.
  \end{cases}
  \label{eq:radial-exact-solution}
\end{equation}
Thus the exact dead core is $B_{1/2}(0)$ and has area $\alpha$.

The unit disk is approximated by inscribed polygonal domains equipped with
conforming Delaunay triangulations. We measure
\begin{align}
  E_{\sobh1}(h)
  &:=
  \Norm{
    \nabla(
      \mathcal E_h\uh-u
    )
  }_{\leb2(\Omega)},
  \label{eq:radial-H1-error}
  \\
  E_\infty(h)
  &:=
  \Norm{
    \mathcal E_h\uh-u
  }_{\leb{\infty}(\Omega)}.
  \label{eq:radial-Linf-error}
\end{align}
For consecutive meshes $h_{j-1}>h_j$, the observed rate is
\begin{equation}
  \operatorname{rate}_E(h_j)
  :=
  \frac{
    \log\qp{
      E(h_{j-1})/E(h_j)
    }
  }{
    \log\qp{
      h_{j-1}/h_j
    }
  }.
  \label{eq:observed-convergence-rate}
\end{equation}

Table~\ref{tab:radial-refinement} shows first-order convergence in the
$H^1$ seminorm and approximately second-order convergence in the maximum
norm. The small iteration counts on the finer meshes result from the
special radial structure. Starting from zero, the first full step solves
a constant-source problem and produces a radially ordered discrete torsion
function. The second full step therefore uses a source approximating
\eqref{eq:radial-exact-source} and produces an iterate already close to the
discrete minimiser. This iterate converges at the third Poisson solve on
all meshes. 

\begin{table}[htbp]
  \centering
\small
\setlength{\tabcolsep}{4pt}
\renewcommand{\arraystretch}{0.92}
\begin{tabular}{crccccc}
\toprule
$h$ & $\dim\Vh$ & $E_{H^1}$ & $\text{rate}_{E_{H^1}}$ & $E_\infty$ & $\text{rate}_{E_\infty}$ & $n_{\mathrm{it}}$ \\
\midrule
$1.127\times10^{-2}$ & $92\,055$ & $3.335\times10^{-3}$ & -- & $2.606\times10^{-5}$ & -- & $3$ \\
$8.116\times10^{-3}$ & $180\,893$ & $2.367\times10^{-3}$ & $1.05$ & $1.292\times10^{-5}$ & $2.14$ & $3$ \\
$5.752\times10^{-3}$ & $369\,852$ & $1.649\times10^{-3}$ & $1.05$ & $6.316\times10^{-6}$ & $2.08$ & $3$ \\
$4.099\times10^{-3}$ & $725\,816$ & $1.174\times10^{-3}$ & $1.00$ & $3.282\times10^{-6}$ & $1.93$ & $3$ \\
$3.572\times10^{-3}$ & $948\,380$ & $1.027\times10^{-3}$ & $0.98$ & $2.493\times10^{-6}$ & $2.00$ & $3$ \\
\bottomrule
\end{tabular}
  \caption{
    Radial refinement study against the exact solution
    \eqref{eq:radial-exact-solution}. The final column is the number of
    Poisson solves required for the stopping criteria.
  }
  \label{tab:radial-refinement}
\end{table}

\subsection{Shaped plasma cross-section}
\label{sec:numerics-shaped}

The second experiment considers a smooth convex cross-section with
elongation and triangularity. Its boundary is described by the support
function
\begin{equation}
  \Gamma(\vartheta)
  =
  a
  \qp{
    1
    +
    \varepsilon_2\cos{2\vartheta}
    +
    \varepsilon_3\cos{3\vartheta}
  },
  \qquad
  0\leq\vartheta<2\pi,
  \label{eq:shaped-support-function}
\end{equation}
through the parametrisation
\begin{align}
  x_1(\vartheta)
  &:=
  R_0
  +
  \Gamma(\vartheta)\cos\vartheta
  -
  \Gamma'(\vartheta)\sin\vartheta,
  \nonumber\\
  x_2(\vartheta)
  &:=
  \Gamma(\vartheta)\sin\vartheta
  +
  \Gamma'(\vartheta)\cos\vartheta.
  \label{eq:shaped-domain}
\end{align}
We take
\begin{equation}
  R_0=1.7,
  \qquad
  a=0.6,
  \qquad
  \varepsilon_2=-0.15,
  \qquad
  \varepsilon_3=0.03.
  \label{eq:shaped-domain-parameters}
\end{equation}
Since
\begin{equation}
  \Gamma(\vartheta)
  +
  \Gamma''(\vartheta)
  \geq
  a
  \qp{
    1
    -
    3\abs{\varepsilon_2}
    -
    8\abs{\varepsilon_3}
  }
  >0,
  \label{eq:shaped-domain-convexity}
\end{equation}
the boundary is strictly convex.

Its area is
\begin{equation}
  A
  =
  \pi a^2
  \qp{
    1
    -
    \frac32\varepsilon_2^2
    -
    4\varepsilon_3^2
  }.
  \label{eq:shaped-domain-area}
\end{equation}
We prescribe
\begin{equation}
  \alpha
  =
  0.2A,
  \qquad
  g(\rho)
  =
  \qp{
    \frac{\rho}{A}
  }^2
  -
  \qp{
    \frac{\alpha}{A}
  }^2,
  \qquad
  0\leq\rho\leq A.
  \label{eq:shaped-source}
\end{equation}
The corresponding clipped rank law is
\begin{equation}
  G(\rho)
  =
  \qb{
    \qp{
      \frac{\rho}{A}
    }^2
    -
    \qp{
      \frac{\alpha}{A}
    }^2
  }_+.
  \label{eq:shaped-clipped-source}
\end{equation}

Each refinement level is initialised independently from zero.

Let
\[
  s_{h,N_h}^{\mathrm{can}}
  :=
  s_h\qb{\overline q_h[u_h^{N_h}]},
  \qquad
  D_h^0
  :=
  \qc{
    x\in\Omega_h:
    s_{h,N_h}^{\mathrm{can}}(x)=0
  }.
  \]
Figure~\ref{fig:shaped-fields} shows the terminal computed equilibrium
$u_h^{N_h}$ and the boundary of $D_h^0$.  The latter is obtained
directly from the zero set of the computed piecewise-affine canonical
source.

\begin{figure}[tbp]
  \centering
  \begin{subfigure}[t]{0.42\textwidth}
    \centering
    \includegraphics[width=\linewidth]
      {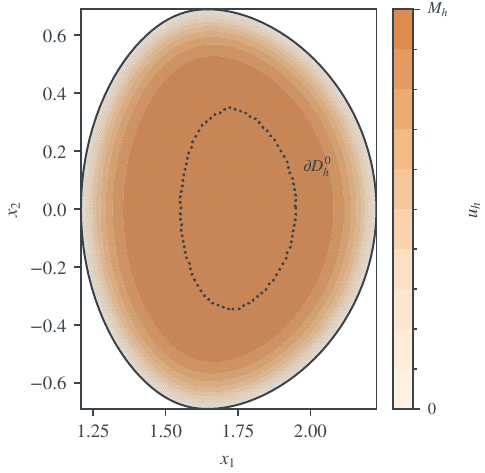}
    \caption{Shaped plasma cross-section on the finest mesh.}
  \end{subfigure}
  \hfill
  \begin{subfigure}[t]{0.54\textwidth}
    \centering
    \includegraphics[width=\linewidth]
      {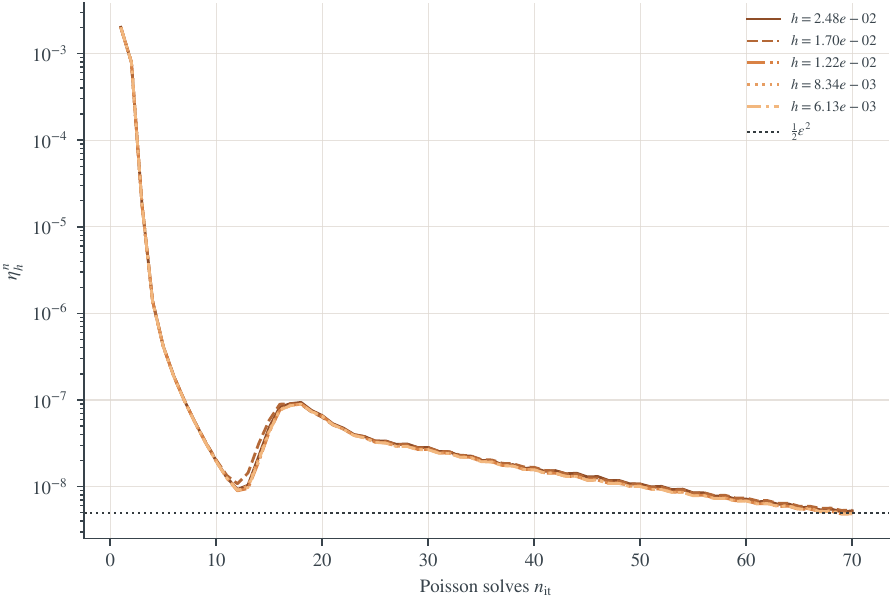}
    \caption{A posteriori energy-gap histories for the shaped refinement study.}
    \label{fig:shaped-aposteriori}
  \end{subfigure}
  \caption{Shaped plasma cross-section and a posteriori energy-gap histories.}
  \label{fig:shaped-fields}
\end{figure}

Table~\ref{tab:shaped-zero-source-area} reports the area of $D_h^0$.
Let $E_D(h) := \norm{|D_h^0|-\alpha}.$ This discrepancy decreases
consistently under refinement, while the number of Poisson solves
remains broadly mesh-independent.

\begin{table}[htbp]
  \centering
\small
\setlength{\tabcolsep}{4pt}
\renewcommand{\arraystretch}{0.92}
\begin{tabular}{crrrrrr}
\toprule
$h$ & $\dim\Vh$ & $|D_h^0|$ & $\alpha$ &
$E_D$ & $\text{rate}_{E_D}$ & $n_{\mathrm{it}}$ \\
\midrule
$6.129\times10^{-3}$ & $88\,902$  & $2.145\times10^{-1}$ & $2.177\times10^{-1}$ & $3.210\times10^{-3}$ & --     & $75$ \\
$4.308\times10^{-3}$ & $174\,519$ & $2.155\times10^{-1}$ & $2.177\times10^{-1}$ & $2.250\times10^{-3}$ & $1.01$ & $70$ \\
$2.985\times10^{-3}$ & $356\,652$ & $2.162\times10^{-1}$ & $2.177\times10^{-1}$ & $1.572\times10^{-3}$ & $0.98$ & $70$ \\
$2.153\times10^{-3}$ & $699\,530$ & $2.166\times10^{-1}$ & $2.177\times10^{-1}$ & $1.125\times10^{-3}$ & $1.02$ & $70$ \\
$1.908\times10^{-3}$ & $913\,893$ & $2.168\times10^{-1}$ & $2.177\times10^{-1}$ & $9.827\times10^{-4}$ & $1.12$ & $70$ \\
\bottomrule
\end{tabular}
\caption{
  Area of the zero canonical-source region for the shaped problem.
  The discrepancy $E_D$ converges at approximately first order.
  The final column is the number of Poisson solves required to satisfy
  the stopping criterion.
}
  \label{tab:shaped-zero-source-area}
\end{table}

Figure~\ref{fig:shaped-aposteriori} shows the decay of $\eta_h^n$ over
the rearrangement iteration for the shaped refinement
sequence. Although no exact solution is available for this geometry,
the experiment provides a non-radial illustration of the method and
shows approximately first-order convergence of the computed
zero-source area towards the prescribed value $\alpha$.

\section{Technical analysis}
\label{sec:proofs}

The proof is organised in four steps. We first collect the
mass-lumping and comparison estimates used throughout. We then analyse
the discrete rearrangement, its supergradients, and the fixed-mesh
iteration. Next we establish consistency for mass-lumped Poisson
problems on the varying polygonal domains and identify the continuum
variational structure. Finally, these ingredients are combined to pass
to the limit in the discrete equilibria and their reconstructed
sources.

\subsection{Lumped quadrature and comparison}

Throughout this section, $I_h$ denotes the nodal interpolation
operator on $\Th$. The following is the standard consistency property
of the nodal quadrature associated with the lumped-mass $\mathbb P_1$
finite element method; see, for example, \cite{ChenThomee1985}. We
record it in measure form because this formulation will be used below.

\begin{lemma}[Convergence of the lumped nodal measures]
\label{lem:nodal-measure-convergence}
For every $\psi\in \cont0(\overline\Omega)$,
\begin{equation}
  \int_{\overline\Omega}
  \psi
  \d\muh
  =
  \sum_{i\in\Nh}
  \omega_i\psi(x_i)
  =
  \int_{\Omega_h}
  I_h\psi\dx.
  \label{eq:nodal-quadrature}
\end{equation}
Consequently,
\begin{equation}
  \muh
  \rightharpoonup
  \mathcal L^2|_\Omega
  \qquad
  \text{weakly as finite measures on }\overline\Omega.
  \label{eq:mu-weak}
\end{equation}
\end{lemma}

\begin{proof}
The first identity follows from the definition of the weights and the
partition of unity:
\[
  \sum_{i\in\Nh}
  \omega_i\psi(x_i)
  =
  \int_{\Omega_h}
  \sum_{i\in\Nh}
  \psi(x_i)\varphi_i\dx
  =
  \int_{\Omega_h}
  I_h\psi\dx.
\]
Since $I_h\psi\to\psi$ uniformly on $\Omega_h$ and
$\abs{\Omega\setminus\Omega_h}\to0$,
\[
  \abs{
    \int_{\overline\Omega}
    \psi\d\muh
    -
    \int_\Omega\psi\dx
  }
  \leq
  \abs{\Omega_h}
  \Norm{I_h\psi-\psi}_{L^\infty(\Omega_h)}
  +
  \int_{\Omega\setminus\Omega_h}
  \abs{\psi}\dx
  \longrightarrow0.
\]
\end{proof}

\begin{lemma}[Lumped norm equivalence]
\label{lem:lumped-norm-equivalence}
There are constants $c,C>0$, independent of $h$, such that
\begin{equation}
  c
  \Norm{v_h}_{L^2(\Omega_h)}^2
  \leq
  \sum_{i\in\Nh}
  \omega_i
  \abs{v_h(x_i)}^2
  \leq
  C
  \Norm{v_h}_{L^2(\Omega_h)}^2
  \label{eq:lumped-L2-equivalence}
\end{equation}
for every $v_h\in\Vht$.
\end{lemma}

\begin{proof}
The corresponding norms are equivalent on the reference triangle.
Scaling to each element and summing proves the result, with uniform
constants provided by shape regularity.
\end{proof}

\begin{lemma}[Mass-lumping error]
\label{lem:mass-lumping-error}
For $p_h,r_h\in\Vht$,
\begin{equation}
  \abs{
    \sum_{i\in\Nh}
    \omega_i
    p_h(x_i)r_h(x_i)
    -
    \int_{\Omega_h}
    p_hr_h\dx
  }
  \leq
  Ch
  \Norm{p_h}_{L^\infty(\Omega_h)}
  \Norm{\nabla r_h}_{L^2(\Omega_h)}.
  \label{eq:mass-lumping-error}
\end{equation}
\end{lemma}

\begin{proof}
On $K\in\Th$, the local quadrature error is $\int_K \bigl(
I_K(p_hr_h)-p_hr_h \bigr) \dx,$ where
\[
  I_K(p_hr_h)(x)-p_h(x)r_h(x) = \sum_{j=1}^3 \lambda_j(x)p_h(x_j)
  \bigl( r_h(x_j)-r_h(x) \bigr).
\]
Since $r_h$ is affine, $\abs{r_h(x_j)-r_h(x)}\leq
Ch_K\abs{\nabla r_h|_K}.$
Thus
\[
  \abs{
    \int_K
    \bigl(
      I_K(p_hr_h)-p_hr_h
    \bigr)
    \dx
  }
  \leq
  Ch_K
  \Norm{p_h}_{L^\infty(K)}
  \abs{K}
  \abs{\nabla r_h|_K}.
\]
Summation and Cauchy--Schwarz give
\eqref{eq:mass-lumping-error}.
\end{proof}

\begin{lemma}[Discrete comparison principle]
\label{lem:discrete-comparison}
If $v_h,w_h\in\Vh$ satisfy
\begin{equation}
  a_h(v_h,\varphi_i)
  \leq
  a_h(w_h,\varphi_i)
  \qquad
  \text{for every }i\in\Nhi,
  \label{eq:comparison-assumption}
\end{equation}
then
\begin{equation}
  v_h(x_i)
  \leq
  w_h(x_i)
  \qquad
  \text{for every }i\in\Nhi.
  \label{eq:comparison-conclusion}
\end{equation}
\end{lemma}

\begin{proof}
In interior nodal coordinates, the hypothesis is $A_hV\leq A_hW$.
Since $A_h^{-1}$ is entrywise nonnegative, $V-W = A_h^{-1}(A_hV-A_hW)
\leq0.$
\end{proof}

\subsection{Discrete rearrangement and supergradients}

Set $m_h := \abs{\Omega_h}.$ A rank allocation is a measurable
partition $(0,m_h) = \bigcup_{i\in\Nh}E_i, \qquad \abs{E_i}=\omega_i,$
up to null sets. It induces the nodal source
\begin{equation}
  q_i
  :=
  \frac{1}{\omega_i}
  \int_{E_i}
  G(r) \d r.
  \label{eq:rank-allocation-source}
\end{equation}

\begin{lemma}[Minimum representation of the discrete rearrangement]
\label{lem:discrete-rearrangement-representation}
For every $v_h\in\Vh$,
\begin{equation}
  \mathcal R_h(v_h)
  =
  \min_{(E_i)}
  \sum_{i\in\Nh}
  v_h(x_i)
  \int_{E_i}
  G(\rho) \d\rho.
  \label{eq:discrete-minimum-representation}
\end{equation}
The minimum assigns the lowest ranks to the largest nodal values. Thus
the nodes on a level $c\in\mathcal L_h(v_h)$ collectively receive the
rank interval
\begin{equation}
  \mathcal I_h(c;v_h).
  \label{eq:minimising-level-rank-interval}
\end{equation}
\end{lemma}

\begin{proof}
Suppose $v_h(x_i)>v_h(x_j)$, but equal-measure sets $A\subset E_i$ and
$B\subset E_j$ satisfy $\int_A G(\rho) \d\rho > \int_B G(\rho)
\d\rho.$ Exchanging $A$ and $B$ changes the pairing by
\[
  \qp{
    v_h(x_i)-v_h(x_j)
  }
  \qp{
    \int_B G(\rho) \d\rho
    -
    \int_A G(\rho) \d\rho
  }
  <
  0.
\]
Hence every minimising allocation is oppositely ordered to the nodal
values. Grouping equal values gives
\[
  \sum_{c\in\mathcal L_h(v_h)}
  c
  \int_{\mathcal I_h(c;v_h)}
  G(\rho) \d\rho
  =
  \mathcal R_h(v_h).
  \]

\end{proof}

\begin{lemma}[Concavity and continuity of $\mathcal R_h$]
\label{lem:discrete-rearrangement-concavity}
The functional $\mathcal R_h:\Vh\to\reals$ is concave and
\begin{equation}
  \abs{
    \mathcal R_h(v_h)-\mathcal R_h(w_h)
  }
  \leq
  M_\alpha
  \sum_{i\in\Nh}
  \omega_i
  \abs{
    v_h(x_i)-w_h(x_i)
  }
  \label{eq:discrete-rearrangement-Lipschitz}
\end{equation}
for all $v_h,w_h\in\Vh$.
\end{lemma}

\begin{proof}
The representation
\eqref{eq:discrete-minimum-representation} expresses $\mathcal R_h$ as
the pointwise minimum of linear functionals, so it is concave. If
$(E_i)$ minimises the representation at $w_h$, then
\[
  \mathcal R_h(v_h)-\mathcal R_h(w_h)
  \leq
  \sum_{i\in\Nh}
  \omega_iq_i
  \bigl(
    v_h(x_i)-w_h(x_i)
  \bigr),
\]
where $0\leq q_i\leq M_\alpha$. Interchanging $v_h$ and $w_h$ proves
\eqref{eq:discrete-rearrangement-Lipschitz}.
\end{proof}

\begin{lemma}[Canonical tied-level source]
\label{lem:canonical-source-supergradient}
Let $c\in\mathcal L_h(v_h)$ and set
\[
  I_c
  :=
  \mathcal I_h(c;v_h),
  \qquad
  \omega_c
  :=
  \abs{I_c}.
  \]
  There is a partition of $I_c$ among the nodes satisfying
$v_h(x_i)=c$ such that
\begin{equation}
  \frac{1}{\omega_i}
  \int_{E_i}G(\rho) \d\rho
  =
  \frac{1}{\omega_c}
  \int_{I_c}G(\rho) \d\rho.
  \label{eq:equal-average-allocation}
\end{equation}
Consequently, the canonical source
$\overline q_h[v_h]$ satisfies
\begin{equation}
  \mathcal R_h(w_h)-\mathcal R_h(v_h)
  \leq
  \sum_{i\in\Nh}
  \omega_i
  \overline q_{h,i}[v_h]
  \bigl(
    w_h(x_i)-v_h(x_i)
  \bigr)
  \label{eq:canonical-source-supergradient}
\end{equation}
for every $w_h\in\Vh$.
\end{lemma}

\begin{proof}
Let
\[
  \mathcal N_c
  :=
  \qc{
    i\in\Nh:
    v_h(x_i)=c
  }
  =
  \qc{
    i_1,\ldots,i_m
  }.
\]
On $I_c$, define the vector measure
\[
  \nu(E)
  :=
  \qp{
    \abs{E},
    \int_E G(\rho)\d\rho
  },
  \qquad
  E\subset I_c
  \text{ measurable}.
\]
Since Lebesgue measure is nonatomic and $G\in L^\infty(I_c)$, the
vector measure $\nu$ is nonatomic. Hence, by the Lyapunov convexity
theorem, the range of the restriction of $\nu$ to any measurable
subset $A\subset I_c$ is convex. In particular, for every
$\lambda\in[0,1]$ there exists a measurable set $E\subset A$ such that
$\nu(E) = \lambda\nu(A).$

We now construct the required partition recursively. Set $A_1:=I_c$
and $W_k:=\sum_{\ell=k}^m\omega_{i_\ell}$. Since $W_1=\omega_c$, we
have $\nu(A_1)=\nu(I_c)=\frac{W_1}{\omega_c}\nu(I_c)$. Suppose that a
measurable remainder $A_k$ has been constructed with
$\nu(A_k)=\frac{W_k}{\omega_c}\nu(I_c)$.
By the preceding convexity property, there exists
$E_{i_k}\subset A_k$ such that
$\nu(E_{i_k})=\frac{\omega_{i_k}}{W_k}\nu(A_k)=
\frac{\omega_{i_k}}{\omega_c}\nu(I_c)$.
Set $A_{k+1} := A_k\setminus E_{i_k}.$ Then $\nu(A_{k+1}) =
\frac{W_{k+1}}{\omega_c}\nu(I_c),$ so the construction may be
continued. Taking the final remainder as $E_{i_m}$ gives a measurable
partition of $I_c$ satisfying
\[
  \abs{E_i}
  =
  \omega_i,
  \qquad
  \int_{E_i}G(\rho)\d\rho
  =
  \frac{\omega_i}{\omega_c}
  \int_{I_c}G(\rho)\d\rho.
\]
Dividing the second identity by $\omega_i$ proves
\eqref{eq:equal-average-allocation}.

Applying this construction independently on every tied nodal level
gives a minimising rank allocation for which $\mathcal R_h(v_h) =
\sum_{i\in\Nh} \omega_i \overline q_{h,i}[v_h]v_h(x_i).$ Using the
same allocation as a competitor in the minimum representation of
$\mathcal R_h(w_h)$ yields
\[
  \mathcal R_h(w_h)-\mathcal R_h(v_h)
  \leq
  \sum_{i\in\Nh}
  \omega_i
  \overline q_{h,i}[v_h]
  \qp{
    w_h(x_i)-v_h(x_i)
  },
\]
which is \eqref{eq:canonical-source-supergradient}.
\end{proof}

\begin{lemma}[Finite-dimensional supergradient representation]
\label{lem:discrete-supergradient-representation}
A functional $\ell_h\in\Vh^\ast$ satisfies
\begin{equation}
  \mathcal R_h(w_h)-\mathcal R_h(v_h)
  \leq
  \ell_h(w_h-v_h)
  \qquad
  \text{for every }w_h\in\Vh
  \label{eq:abstract-discrete-supergradient}
\end{equation}
if and only if there is an admissible rearrangement source
$q_h=(q_{h,i})_{i\in\Nh}$ at $v_h$ such that
\begin{equation}
  \ell_h(w_h)
  =
  \sum_{i\in\Nh}
  \omega_iq_{h,i}w_h(x_i)
  \qquad
  \text{for every }w_h\in\Vh.
  \label{eq:represented-discrete-supergradient}
\end{equation}
In particular, $0\leq q_{h,i}\leq M_\alpha$.
\end{lemma}

\begin{proof}
Let $\mathfrak S_h$ be the set of permutations of $\Nh$. For
$\pi\in\mathfrak S_h$, define $r_{\pi,0}:=0$ and
$r_{\pi,k}:=\sum_{\ell=1}^k\omega_{\pi(\ell)}$
and
\begin{equation}
  q_{\pi,\pi(k)}
  :=
  \frac{1}{\omega_{\pi(k)}}
  \int_{r_{\pi,k-1}}^{r_{\pi,k}}
  G(\rho) \d\rho.
  \label{eq:permutation-source}
\end{equation}
With $L_\pi(w_h):=\sum_{i\in\Nh}\omega_iq_{\pi,i}w_h(x_i)$,
Lemma~\ref{lem:discrete-rearrangement-representation} gives
\begin{equation}
  \mathcal R_h(w_h)
  =
  \min_{\pi\in\mathfrak S_h}
  L_\pi(w_h).
  \label{eq:finite-permutation-representation}
\end{equation}
If $\mathfrak S_h(v_h):=\qc{\pi:L_\pi(v_h)=\mathcal R_h(v_h)}$,
then the superdifferential of $\mathcal R_h$ at $v_h$ is
\begin{equation}
  \partial^+\mathcal R_h(v_h)
  =
  \operatorname{co}
  \qc{
    L_\pi:
    \pi\in\mathfrak S_h(v_h)
  },
  \label{eq:finite-superdifferential}
\end{equation}
where $\operatorname{co}$ denotes the convex hull. Hence
$\ell_h=\sum_{\pi\in\mathfrak S_h(v_h)}\lambda_\pi L_\pi$, where
$\lambda_\pi\geq0$ and $\sum_\pi\lambda_\pi=1$. The coefficients
$q_{h,i}:=\sum_{\pi}\lambda_\pi q_{\pi,i}$ satisfy
$0\leq q_{h,i}\leq M_\alpha$
and
\[
  \ell_h(w_h)
  =
  \sum_{i\in\Nh}
  \omega_iq_{h,i}w_h(x_i)
  \qquad
  \text{for every }w_h\in\Vh.
\]
Since $\ell_h$ is a supergradient at $v_h$, the vector $q_h$ therefore
satisfies
\eqref{eq:admissible-source-inequality} and is an admissible
rearrangement source.

Conversely, suppose that $q_h$ is an admissible rearrangement source
at $v_h$, and define
\[
  \ell_h(w_h)
  :=
  \sum_{i\in\Nh}
  \omega_iq_{h,i}w_h(x_i).
\]
Then \eqref{eq:admissible-source-inequality} is precisely
\[
  \mathcal R_h(w_h)-\mathcal R_h(v_h)
  \leq
  \ell_h(w_h-v_h)
  \qquad
  \text{for every }w_h\in\Vh.
\]
Hence $\ell_h\in\partial^+\mathcal R_h(v_h)$, completing the
equivalence.
\end{proof}

\subsection{Discrete well-posedness and iteration}

\begin{lemma}[Coercivity of the discrete energy]
\label{lem:discrete-energy-coercivity}
There is a constant $C>0$, independent of $h$ and $v_h$, such that
\begin{equation}
  \mathcal J_h(v_h)
  \geq
  \frac12
  \Norm{\nabla v_h}_{L^2(\Omega_h)}^2
  -
  C
  \Norm{\nabla v_h}_{L^2(\Omega_h)}.
  \label{eq:discrete-energy-coercivity}
\end{equation}
\end{lemma}

\begin{proof}
By Lemmas~\ref{lem:discrete-rearrangement-concavity}
and~\ref{lem:lumped-norm-equivalence},
\[
  \abs{\mathcal R_h(v_h)}
  \leq
  M_\alpha
  \sum_{i\in\Nh}
  \omega_i\abs{v_h(x_i)}
  \leq
  C\Norm{v_h}_{L^2(\Omega_h)}.
\]
The Poincaré inequality for the zero extension to $\Omega$ gives
$\abs{\mathcal R_h(v_h)}\leq C\Norm{\nabla v_h}_{L^2(\Omega_h)}$,
which proves the claim.
\end{proof}

\begin{lemma}[Discrete existence, uniqueness and optimality]
\label{lem:discrete-well-posedness}
The functional $\mathcal J_h$ has a unique minimiser $\uh\in\Vh$.
Moreover, the following are equivalent:
\begin{enumerate}
  \item
  $\uh$ minimises $\mathcal J_h$;

  \item
  \begin{equation}
    a_h(\uh,v_h-\uh)
    \geq
    \mathcal R_h(v_h)-\mathcal R_h(\uh)
    \qquad
    \text{for every }v_h\in\Vh;
    \label{eq:technical-discrete-VI}
  \end{equation}

  \item
  there is an admissible source
  $q_h^\ast=(q_{h,i}^\ast)_{i\in\Nh}$ at $\uh$ such that
  \begin{equation}
    a_h(\uh,v_h)
    =
    \sum_{i\in\Nh}
    \omega_iq_{h,i}^\ast v_h(x_i)
    \qquad
    \text{for every }v_h\in\Vh.
    \label{eq:technical-discrete-source-equation}
  \end{equation}
\end{enumerate}
\end{lemma}

\begin{proof}
The functional $-\mathcal R_h$ is convex and continuous, while
$v_h\longmapsto\frac12a_h(v_h,v_h)$ is strictly convex on $\Vh$.
Hence $\mathcal J_h$ is strictly convex.
By Lemma~\ref{lem:discrete-energy-coercivity}, it is also coercive.
Since $\Vh$ is finite dimensional, $\mathcal J_h$ therefore has a
unique minimiser.

We prove the equivalence in the order
$(1)\Longrightarrow(3)\Longrightarrow(2)\Longrightarrow(1)$.

Suppose first that $\uh$ minimises $\mathcal J_h$. The
finite-dimensional convex optimality condition gives a supergradient
$\ell_h\in\partial^+\mathcal R_h(\uh)$ satisfying
\begin{equation}
  a_h(\uh,v_h)
  =
  \ell_h(v_h)
  \qquad
  \text{for every }v_h\in\Vh.
  \label{eq:technical-optimality-balance}
\end{equation}
By Lemma~\ref{lem:discrete-supergradient-representation}, there is an
admissible source $q_h^\ast$ at $\uh$ representing $\ell_h$. Hence
\[
  a_h(\uh,v_h)
  =
  \sum_{i\in\Nh}
  \omega_iq_{h,i}^\ast v_h(x_i)
  \qquad
  \text{for every }v_h\in\Vh,
\]
which proves $(1)\Rightarrow(3)$.

Now suppose that $(3)$ holds. Since $q_h^\ast$ is admissible at $\uh$,
\[
  \mathcal R_h(v_h)-\mathcal R_h(\uh)
  \leq
  \sum_{i\in\Nh}
  \omega_iq_{h,i}^\ast
  \qp{
    v_h(x_i)-\uh(x_i)
  }.
\]
Using \eqref{eq:technical-discrete-source-equation} with
$v_h-\uh$ therefore gives $\mathcal R_h(v_h)-\mathcal R_h(\uh)
\leq a_h(\uh,v_h-\uh)$,
which is \eqref{eq:technical-discrete-VI}. Thus
$(3)\Rightarrow(2)$.

Finally, suppose that $(2)$ holds. Expanding the quadratic term gives
\begin{align*}
  \mathcal J_h(v_h)-\mathcal J_h(\uh)
  &=
  a_h(\uh,v_h-\uh)
  +
  \frac12
  a_h(v_h-\uh,v_h-\uh)
  \\
  &\qquad
  -
  \qp{
    \mathcal R_h(v_h)-\mathcal R_h(\uh)
  }.
\end{align*}
By \eqref{eq:technical-discrete-VI},
\[
  \mathcal J_h(v_h)-\mathcal J_h(\uh)
  \geq
  \frac12
  a_h(v_h-\uh,v_h-\uh)
  \geq
  0.
\]
Hence $\uh$ minimises $\mathcal J_h$, proving $(2)\Rightarrow(1)$.
The same inequality also recovers uniqueness from strict convexity.
\end{proof}

\begin{lemma}[Discrete barrier bound]
\label{lem:discrete-barrier-bound}
The discrete minimiser satisfies
\begin{equation}
  \uh\in K_h.
  \label{eq:technical-discrete-invariant-bound}
\end{equation}
\end{lemma}

\begin{proof}
For $i\in\Nhi$, $a_h(\uh,\varphi_i)=\omega_iq_{h,i}^\ast$,
where $0\leq q_{h,i}^\ast\leq M_\alpha$. Hence
$0\leq a_h(\uh,\varphi_i)\leq M_\alpha\omega_i=
a_h(z_h,\varphi_i)$.
Lemma~\ref{lem:discrete-comparison} applied twice gives
$0\leq\uh(x_i)\leq z_h(x_i)$ for every $i\in\Nhi$.
\end{proof}

\begin{lemma}[Invariance of the rearrangement iteration]
\label{lem:rearrangement-iteration-invariance}
If $u_h^n\in K_h$, then
$\widehat u_h^{n+1}\in K_h$. Consequently, $u_h^0\in K_h$ implies
$u_h^n\in K_h$ for every $n\geq0$.
\end{lemma}

\begin{proof}
Since $0\leq q_{h,i}^n\leq M_\alpha$,
$0\leq a_h(\widehat u_h^{n+1},\varphi_i)\leq
a_h(z_h,\varphi_i)$.
The comparison principle gives
$\widehat u_h^{n+1}\in K_h$. The set $K_h$ is convex, so the relaxed
update also belongs to $K_h$.
\end{proof}

\begin{lemma}[Convergence of the rearrangement iteration]
\label{lem:rearrangement-iteration-convergence}
Let $\uh$ be the minimiser of $\mathcal J_h$. If
$u_h^0\in K_h$ and \eqref{eq:step-size-conditions} holds, then
\begin{equation}
  \Norm{
    \nabla(
      u_h^n-\uh
    )
  }_{L^2(\Omega_h)}
  \longrightarrow0.
  \label{eq:technical-iteration-H1-convergence}
\end{equation}
For fixed $h$,
\begin{equation}
  \Norm{
    u_h^n-\uh
  }_{L^\infty(\Omega_h)}
  \longrightarrow0.
  \label{eq:technical-iteration-Linf-convergence}
\end{equation}
\end{lemma}

\begin{proof}
Write
\[
  \Norm{v_h}_{a_h}^2
  :=
  a_h(v_h,v_h),
  \qquad
  d_h^n
  :=
  u_h^n-\widehat u_h^{n+1}.
\]
The Poisson step gives
\begin{equation}
  a_h(d_h^n,v_h)
  =
  a_h(u_h^n,v_h)
  -
  \sum_{i\in\Nh}
  \omega_iq_{h,i}^nv_h(x_i).
  \label{eq:iteration-subgradient-identity}
\end{equation}
By Lemma~\ref{lem:canonical-source-supergradient},
$q_h^n=\overline q_h[u_h^n]$ represents a supergradient of
$\mathcal R_h$ at $u_h^n$, and hence
\[
  \mathcal R_h(v_h)-\mathcal R_h(u_h^n)
  \leq
  \sum_{i\in\Nh}
  \omega_iq_{h,i}^n
  \qp{
    v_h(x_i)-u_h^n(x_i)
  }.
\]
Using \eqref{eq:iteration-subgradient-identity} with
$v_h-u_h^n$ in place of $v_h$ therefore gives
\[
\begin{aligned}
  \mathcal J_h(v_h)-\mathcal J_h(u_h^n)
  &\geq
  \frac12
  \qp{
    \Norm{v_h}_{a_h}^2
    -
    \Norm{u_h^n}_{a_h}^2
  }
  -
  a_h(
    u_h^n,
    v_h-u_h^n
  )
  +
  a_h(
    d_h^n,
    v_h-u_h^n
  )
  \\
  &=
  \frac12
  \Norm{v_h-u_h^n}_{a_h}^2
  +
  a_h(
    d_h^n,
    v_h-u_h^n
  ).
\end{aligned}
\]
Hence
\begin{equation}
  \mathcal J_h(v_h)
  \geq
  \mathcal J_h(u_h^n)
  +
  a_h(d_h^n,v_h-u_h^n)
  +
  \frac12
  \Norm{v_h-u_h^n}_{a_h}^2.
  \label{eq:strong-subgradient-inequality}
\end{equation}

Taking $v_h=\uh$ and using the minimality
$\mathcal J_h(\uh)\leq\mathcal J_h(u_h^n)$ yields
\begin{equation}
  a_h(
    d_h^n,
    u_h^n-\uh
  )
  \geq
  \frac12
  \Norm{u_h^n-\uh}_{a_h}^2.
  \label{eq:subgradient-distance-control}
\end{equation}

Since
\[
  u_h^{n+1}
  =
  u_h^n-\theta_nd_h^n,
\]
we have
\[
\begin{aligned}
  \Norm{u_h^{n+1}-\uh}_{a_h}^2
  &=
  \Norm{
    u_h^n-\uh-\theta_nd_h^n
  }_{a_h}^2
  \\
  &=
  \Norm{u_h^n-\uh}_{a_h}^2
  -
  2\theta_n
  a_h(
    d_h^n,
    u_h^n-\uh
  )
  +
  \theta_n^2
  \Norm{d_h^n}_{a_h}^2.
\end{aligned}
\]
Applying \eqref{eq:subgradient-distance-control} to the middle term
gives
\[
  \Norm{u_h^{n+1}-\uh}_{a_h}^2
  \leq
  (1-\theta_n)
  \Norm{u_h^n-\uh}_{a_h}^2
  +
  \theta_n^2
  \Norm{d_h^n}_{a_h}^2.
\]

By Lemma~\ref{lem:rearrangement-iteration-invariance}, the two points
$u_h^n$ and $\widehat u_h^{n+1}$ remain in the compact set $K_h$.
Thus, for fixed $h$,
\[
  \Norm{d_h^n}_{a_h}
  =
  \Norm{
    u_h^n-\widehat u_h^{n+1}
  }_{a_h}
  \leq
  C_h.
\]
Setting $e_n:=\Norm{u_h^n-\uh}_{a_h}^2$ gives
\begin{equation}
  e_{n+1}
  \leq
  (1-\theta_n)e_n
  +
  C_h^2\theta_n^2.
  \label{eq:iteration-recursion}
\end{equation}
We now show directly that $e_n\to0$. Fix $N\geq0$. Iterating
\eqref{eq:iteration-recursion} gives
\begin{align}
  e_{n+1}
  &\leq
  e_N
  \prod_{k=N}^{n}
  \qp{
    1-\theta_k
  }
  \nonumber\\
  &\quad
  +
  C_h^2
  \sum_{j=N}^{n}
  \theta_j^2
  \prod_{k=j+1}^{n}
  \qp{
    1-\theta_k
  }.
  \label{eq:iteration-recursion-product}
\end{align}
Since $0<\theta_k\leq1$ and $\sum_k\theta_k=\infty$,
$\prod_{k=N}^{n}\qp{1-\theta_k}\longrightarrow0$ as $n\to\infty$.
Moreover, every product in the sum is bounded by one, so
$\limsup_{n\to\infty}e_n\leq C_h^2\sum_{j=N}^{\infty}\theta_j^2$.
Because $\sum_j\theta_j^2<\infty$, the right-hand side tends to zero
as $N\to\infty$. Hence $e_n\to0$, proving
\eqref{eq:technical-iteration-H1-convergence}.

The $L^\infty$ conclusion follows from equivalence of norms on the
fixed finite-dimensional space $\Vh$.
\end{proof}

\begin{proof}[Proof of
Theorem~\ref{thm:discrete-energy-certificate}]
Fix $0\leq j\leq n$. Since $q_h^j$ represents a supergradient of
$\mathcal R_h$ at $u_h^j$,
\begin{equation}
  \mathcal R_h(v_h)
  -
  \mathcal R_h(u_h^j)
  \leq
  \sum_{i\in\mathcal N_h}
  \omega_iq_{h,i}^j
  \qp{
    v_h(x_i)-u_h^j(x_i)
  }
  \label{eq:certificate-supergradient-bound}
\end{equation}
for every $v_h\in\Vh$. By the defining Poisson equation for
$\widehat u_h^{j+1}$, the right-hand side equals
$a_h\qp{\widehat u_h^{j+1},v_h-u_h^j}$.
It follows that
\begin{align}
  \mathcal J_h(v_h)
  &=
  \frac12a_h(v_h,v_h)
  -
  \mathcal R_h(v_h)
  \nonumber\\
  &\geq
  \frac12a_h(v_h,v_h)
  -
  \mathcal R_h(u_h^j)
  -
  a_h\qp{
    \widehat u_h^{j+1},
    v_h-u_h^j
  }
  \nonumber\\
  &=
  \frac12
  a_h\qp{
    v_h-\widehat u_h^{j+1},
    v_h-\widehat u_h^{j+1}
  }
  +
  \ell_h^j.
  \label{eq:certificate-completed-square}
\end{align}
Hence $\mathcal J_h(v_h)\geq\ell_h^j$ for every $v_h\in\Vh$.
Taking $v_h=u_h$ gives $\ell_h^j\leq\mathcal J_h^\star$. Since this
holds for every $0\leq j\leq n$,
$L_h^n=\max_{0\leq j\leq n}\ell_h^j\leq\mathcal J_h^\star$. On the
other hand, the minimality of $u_h$ gives
$\mathcal J_h^\star\leq\mathcal J_h(u_h^n)=\mathcal J_h^n$.
This proves
\eqref{eq:certified-energy-bracket}, and therefore
\[
  0
  \leq
  \mathcal J_h^n-\mathcal J_h^\star
  \leq
  \mathcal J_h^n-L_h^n
  =
  \eta_h^n,
\]
which is \eqref{eq:certified-energy-error}.

Finally, the strong convexity estimate from
Lemma~\ref{lem:discrete-well-posedness} gives
\[
  \mathcal J_h(u_h^n)
  -
  \mathcal J_h(u_h)
  \geq
  \frac12
  a_h\qp{
    u_h^n-u_h,
    u_h^n-u_h
  }.
\]
Combining this inequality with
\eqref{eq:certified-energy-error} yields
\[
  \Norm{
    \nabla(u_h^n-u_h)
  }_{\leb2(\Omega_h)}
  \leq
  \sqrt{2\eta_h^n},
\]
which is \eqref{eq:certified-energy-norm-error}.
\end{proof}

\subsection{Linear consistency}

We first record the uniform estimates used below. Under convexity of
$\Omega_h$, uniform shape regularity and quasi-uniformity, there is a
constant $C$, independent of $h$, such that
\begin{align}
  \Norm{v_h}_{L^\infty(\Omega_h)}
  &\leq
  C
  \bigl(
    1+\abs{\log h}
  \bigr)^{1/2}
  \Norm{\nabla v_h}_{L^2(\Omega_h)},
  \label{eq:uniform-discrete-Sobolev}
  \\
  \Norm{
    \nabla(
      w-I_hw
    )
  }_{L^2(\Omega_h)}
  &\leq
  Ch
  \Norm{w}_{H^2(\Omega_h)},
  \label{eq:uniform-H1-interpolation}
  \\
  \Norm{
    w-I_hw
  }_{L^\infty(\Omega_h)}
  &\leq
  Ch
  \Norm{w}_{H^2(\Omega_h)}.
  \label{eq:uniform-Linf-interpolation}
\end{align}
These are the standard two-dimensional discrete Sobolev and scaled
interpolation estimates; see
\cite{RannacherScott1982,BrennerScott2008}.

\begin{lemma}[Linear consistency for mass-lumped Poisson problems]
\label{lem:linear-compactness-consistency}
Let $h\to0$ index a family of polygonal domains and meshes satisfying
\[
  \Omega_{-\delta_h}
  \subset
  \Omega_h
  \subset
  \Omega,
  \qquad
  \delta_h\longrightarrow0,
\]
where
\[
  \Omega_{-\delta}
  :=
  \qc{
    x\in\Omega:
    \dist(x,\partial\Omega)>\delta
  }.
\]
Assume that the meshes are uniformly shape-regular and quasi-uniform.
Let $q_h=(q_{h,i})_{i\in\Nh}$ satisfy
\begin{equation}
  \sup_h
  \max_{i\in\Nh}
  \abs{q_{h,i}}
  <\infty,
  \label{eq:uniform-nodal-source-bound}
\end{equation}
and define
\begin{equation}
  s_h
  :=
  s_h[q_h]
  =
  \sum_{i\in\Nh}
  q_{h,i}\varphi_i.
  \label{eq:linear-reconstructed-source}
\end{equation}
For any measurable function $r_h:\Omega_h\to\mathbb R$, we denote its
zero extension to $\Omega$ by
\begin{equation}
  \widetilde r_h(x)
  :=
  \begin{cases}
    r_h(x),
      &x\in\Omega_h,
    \\
    0,
      &x\in\Omega\setminus\Omega_h.
  \end{cases}
  \label{eq:source-zero-extension}
\end{equation}
This notation will be used in particular for reconstructed sources,
which need not have zero trace on $\partial\Omega_h$.

Let $U_h\in\Vh$ solve
\begin{equation}
  a_h(U_h,v_h)
  =
  \sum_{i\in\Nh}
  \omega_iq_{h,i}v_h(x_i)
  \qquad
  \text{for every }v_h\in\Vh.
  \label{eq:mass-lumped-linear-problem}
\end{equation}
If
\begin{equation}
  \widetilde{s_h}
  \stackrel{*}{\rightharpoonup}
  f
  \qquad
  \text{in }\leb{\infty}(\Omega)
  \quad\text{as }h\to0,
  \label{eq:reconstructed-source-weak-star}
\end{equation}
then, as $h\to0$,
\begin{align}
  \mathcal E_hU_h
  &\longrightarrow w
  &&\text{strongly in }\hoz(\Omega),
  \label{eq:linear-domain-H1-convergence}
  \\
  \mathcal E_hU_h
  &\longrightarrow w
  &&\text{strongly in }\leb{\infty}(\Omega),
  \label{eq:linear-domain-Linf-convergence}
\end{align}
where $w\in \hoz(\Omega)$ is the solution of
\begin{equation}
  \int_\Omega
  \nabla w\cdot\nabla\phi\dx
  =
  \int_\Omega
  f\phi\dx
  \qquad
  \text{for every }\phi\in \hoz(\Omega).
  \label{eq:limit-domain-poisson}
\end{equation}
\end{lemma}

\begin{proof}
The nodal bound and the partition of unity imply
\begin{equation}
  \sup_h
  \Norm{s_h}_{L^\infty(\Omega_h)}
  <\infty.
  \label{eq:technical-uniform-source-bound}
\end{equation}

Let $w(h)\in \hoz(\Omega_h)$ solve
\begin{equation}
  \int_{\Omega_h}
  \nabla w(h)\cdot\nabla\phi\dx
  =
  \int_{\Omega_h}
  s_h\phi\dx
  \qquad
  \text{for every }\phi\in \hoz(\Omega_h).
  \label{eq:technical-continuous-Poisson}
\end{equation}
Testing with $w(h)$ and using Poincaré's inequality gives
\begin{equation}
  \sup_h
  \Norm{\nabla w(h)}_{L^2(\Omega_h)}
  <\infty.
  \label{eq:technical-continuous-energy-bound}
\end{equation}

Let $z(h)\in \hoz(\Omega)$ solve the corresponding problem on the
fixed domain:
\begin{equation}
  \int_\Omega
  \nabla z(h)\cdot\nabla\phi\dx
  =
  \int_\Omega
  \widetilde{s_h}\phi\dx
  \qquad
  \text{for every }\phi\in \hoz(\Omega).
  \label{eq:technical-fixed-domain-problem}
\end{equation}
By \eqref{eq:technical-uniform-source-bound}, the zero extensions satisfy
\[
  \sup_h
  \Norm{
    \widetilde{s_h}
  }_{L^p(\Omega)}
  <\infty
  \qquad
  \text{for every }1\leq p<\infty.
\]
Since the fixed domain $\Omega$ has $\cont{1,1}$ boundary, elliptic
regularity therefore gives, for every fixed $p>2$,
\begin{equation}
  \Norm{z(h)}_{W^{2,p}(\Omega)}
  \leq
  C_p
  \Norm{
    \widetilde{s_h}
  }_{L^p(\Omega)}
  \leq
  C_p.
  \label{eq:technical-fixed-domain-W2p}
\end{equation}
In two dimensions,
\[
  W^{2,p}(\Omega)
  \hookrightarrow
  C^{1,\gamma}(\overline\Omega),
  \qquad
  \gamma
  =
  1-\frac{2}{p}
  >0,
\]
and the embedding
$C^{1,\gamma}(\overline\Omega)\hookrightarrow C^1(\overline\Omega)$
is compact. Hence $\{z(h)\}_h$ is precompact in
$\cont1(\overline\Omega)$.

If $z(h_k)\to z$ strongly in $\cont1(\overline\Omega)$ along a
subsequence, then passing to the limit in
\eqref{eq:technical-fixed-domain-problem} using
\eqref{eq:reconstructed-source-weak-star} shows that $z$ satisfies
\eqref{eq:limit-domain-poisson}. By uniqueness, $z=w$. Since every
convergent subsequence has the same limit, the full sequence satisfies
\begin{equation}
  z(h)
  \longrightarrow
  w
  \qquad
  \text{strongly in }\cont1(\overline\Omega).
  \label{eq:technical-fixed-domain-C1}
\end{equation}

We next compare the fixed-domain and varying-domain solutions. The
inclusion
\[
  \Omega_{-\delta_h}
  \subset
  \Omega_h
\]
implies
\begin{equation}
  \sup_{x\in\partial\Omega_h}
  \dist(x,\partial\Omega)
  \leq
  \delta_h.
  \label{eq:technical-boundary-distance}
\end{equation}
Indeed, if $x\in\partial\Omega_h$ satisfied
$\dist(x,\partial\Omega)>\delta_h$, then
$x\in\Omega_{-\delta_h}\subset\Omega_h$, contradicting
$x\in\partial\Omega_h$.

Since $\widetilde{s_h}=s_h$ on $\Omega_h$, the function
\[
  w(h)-z(h)|_{\Omega_h}
\]
is harmonic in $\Omega_h$ and has boundary value $-z(h)$ on
$\partial\Omega_h$. The maximum principle, the zero trace of $z(h)$ on
$\partial\Omega$, and \eqref{eq:technical-fixed-domain-W2p} give
\begin{align}
  \Norm{
    w(h)-z(h)
  }_{L^\infty(\Omega_h)}
  &\leq
  \Norm{z(h)}_{L^\infty(\partial\Omega_h)}
  \notag
  \\
  &\leq
  \Norm{\nabla z(h)}_{\leb{\infty}(\Omega)}
  \sup_{x\in\partial\Omega_h}
  \dist(x,\partial\Omega)
  \notag
  \\
  &\leq
  C\delta_h.
  \label{eq:technical-harmonic-comparison}
\end{align}
It follows from \eqref{eq:technical-fixed-domain-C1} that
\begin{equation}
  \Norm{
    w(h)-w
  }_{L^\infty(\Omega_h)}
  \longrightarrow0.
  \label{eq:technical-interior-Linf}
\end{equation}

It remains to control the part of $\Omega$ omitted by $\Omega_h$.
The domain inclusions give
\[
  \Omega\setminus\Omega_h
  \subset
  \qc{
    x\in\Omega:
    \dist(x,\partial\Omega)\leq\delta_h
  }.
\]
Since $w\in \cont1(\overline\Omega)$ and $w=0$ on $\partial\Omega$,
\begin{equation}
  \Norm{w}_{L^\infty(\Omega\setminus\Omega_h)}
  \leq
  \delta_h
  \Norm{\nabla w}_{\leb{\infty}(\Omega)}
  \longrightarrow0.
  \label{eq:technical-boundary-strip}
\end{equation}
Because $\mathcal E_hw(h)=0$ on $\Omega\setminus\Omega_h$,
\eqref{eq:technical-interior-Linf} and
\eqref{eq:technical-boundary-strip} yield
\begin{equation}
  \mathcal E_hw(h)
  \longrightarrow
  w
  \qquad
  \text{strongly in }\leb{\infty}(\Omega).
  \label{eq:technical-continuous-Linf}
\end{equation}

The energy identity gives
\[
  \Norm{\nabla w(h)}_{L^2(\Omega_h)}^2
  =
  \int_\Omega
  \widetilde{s_h}
  \mathcal E_hw(h)
  \dx.
\]
By \eqref{eq:technical-continuous-Linf}, the uniform source bound and
the weak-star convergence of $\widetilde{s_h}$,
\[
  \int_\Omega
  \widetilde{s_h}
  \mathcal E_hw(h)
  \dx
  \longrightarrow
  \int_\Omega
  fw\dx
  =
  \Norm{\nabla w}_{L^2(\Omega)}^2.
\]
Together with weak compactness in $\hoz(\Omega)$, this proves
\begin{equation}
  \mathcal E_hw(h)
  \longrightarrow
  w
  \qquad
  \text{strongly in }\hoz(\Omega).
  \label{eq:technical-continuous-H1}
\end{equation}

Next, let $W_h\in\Vh$ solve
\begin{equation}
  a_h(W_h,v_h)
  =
  \int_{\Omega_h}
  s_hv_h\dx
  \qquad
  \text{for every }v_h\in\Vh.
  \label{eq:technical-consistent-fem-problem}
\end{equation}
Convexity of $\Omega_h$, the Miranda--Talenti estimate and Poincaré's
inequality give
\begin{equation}
  \Norm{w(h)}_{H^2(\Omega_h)}
  \leq
  C
  \Norm{s_h}_{L^2(\Omega_h)}
  \leq C.
  \label{eq:technical-uniform-H2}
\end{equation}
Céa's lemma and \eqref{eq:uniform-H1-interpolation} imply
\begin{equation}
  \Norm{
    \nabla(
      W_h-w(h)
    )
  }_{L^2(\Omega_h)}
  \leq
  Ch.
  \label{eq:technical-Ritz-H1}
\end{equation}
Using \eqref{eq:uniform-discrete-Sobolev} on
$W_h-I_hw(h)$, together with
\eqref{eq:uniform-Linf-interpolation}, gives
\begin{equation}
  \Norm{
    W_h-w(h)
  }_{L^\infty(\Omega_h)}
  \leq
  Ch
  \bigl(
    1+\abs{\log h}
  \bigr)^{1/2}.
  \label{eq:technical-Ritz-Linf}
\end{equation}

Finally, set $e_h:=U_h-W_h$. Then
\begin{equation}
  a_h(e_h,v_h)
  =
  \sum_{i\in\Nh}
  \omega_iq_{h,i}v_h(x_i)
  -
  \int_{\Omega_h}
  s_hv_h\dx.
  \label{eq:technical-lumping-error-equation}
\end{equation}
Since $s_h(x_i)=q_{h,i}$, testing with $e_h$ and applying
Lemma~\ref{lem:mass-lumping-error} gives
\begin{equation}
  \Norm{\nabla e_h}_{L^2(\Omega_h)}
  \leq
  Ch
  \Norm{s_h}_{L^\infty(\Omega_h)}
  \longrightarrow0.
  \label{eq:technical-lumping-H1}
\end{equation}
The discrete Sobolev estimate gives
\begin{equation}
  \Norm{e_h}_{L^\infty(\Omega_h)}
  \leq
  Ch
  \bigl(
    1+\abs{\log h}
  \bigr)^{1/2}
  \Norm{s_h}_{L^\infty(\Omega_h)}
  \longrightarrow0.
  \label{eq:technical-lumping-Linf}
\end{equation}
Combining the continuous-domain, Ritz and quadrature estimates proves
\eqref{eq:linear-domain-H1-convergence} and
\eqref{eq:linear-domain-Linf-convergence}.
\end{proof}

\subsection{Continuum variational structure}

For $v\in L^1(\Omega)$, let $v^\ast$ denote the decreasing
rearrangement defined in
\eqref{eq:continuum-decreasing-rearrangement}, and define
\begin{equation}
  \mathcal R(v)
  :=
  \int_0^{\abs{\Omega}}
  G(r)v^\ast(r)\d r.
  \label{eq:continuum-rearrangement-functional}
\end{equation}
We write $q\sim G$ if $q:\Omega\to\mathbb R$ and
$G:(0,\abs{\Omega})\to\mathbb R$ are equimeasurable, that is,
\begin{equation}
  \abs{
    \qc{
      x\in\Omega:
      q(x)>t
    }
  }
  =
  \abs{
    \qc{
      r\in(0,\abs{\Omega}):
      G(r)>t
    }
  }
  \qquad
  \text{for every }t\in\mathbb R.
  \label{eq:continuum-equimeasurability}
\end{equation}
Thus $q$ and $G$ have the same distribution of values.

Define
\begin{equation}
  \mathcal J(v)
  :=
  \frac12
  \int_\Omega
  \abs{\nabla v}^2\dx
  -
  \mathcal R(v),
  \qquad
  v\in \hoz(\Omega).
  \label{eq:continuum-rearrangement-energy}
\end{equation}

\begin{lemma}[Continuum rearrangement functional]
\label{lem:continuum-rearrangement-properties}
The functional $\mathcal R:L^1(\Omega)\to\reals$ is concave and
satisfies
\begin{equation}
  \abs{
    \mathcal R(v)-\mathcal R(w)
  }
  \leq
  M_\alpha
  \Norm{v-w}_{L^1(\Omega)}
  \qquad
  \text{for every }v,w\in L^1(\Omega).
  \label{eq:continuum-rearrangement-Lipschitz}
\end{equation}
Consequently, $\mathcal J$ is strictly convex on $\hoz(\Omega)$.
\end{lemma}

\begin{proof}
Since $G$ is nondecreasing, the Hardy--Littlewood rearrangement
inequality and the characterisation of equimeasurable rearrangements
\cite{HardyLittlewoodPolya1952,Ryff1965},
\begin{equation}
  \mathcal R(v)
  =
  \inf_{q\sim G}
  \int_\Omega vq\dx.
  \label{eq:continuum-minimum-representation}
\end{equation}
Indeed, among all spatial rearrangements of $G$, the integral is
minimised by pairing the decreasing rearrangement of $v$ with the
increasing rearrangement of $G$, which is $G$ itself.

For every fixed $q\sim G$, the map
\[
  v
  \longmapsto
  \int_\Omega
  vq\dx
\]
is linear. Hence \eqref{eq:continuum-minimum-representation} expresses
$\mathcal R$ as an infimum of linear functionals, and therefore
$\mathcal R$ is concave.

Equimeasurability preserves the essential range. Since
\[
  0
  \leq
  G
  \leq
  M_\alpha
  \qquad
  \text{almost everywhere on }(0,\abs{\Omega}),
\]
every $q\sim G$ satisfies
\begin{equation}
  0
  \leq
  q
  \leq
  M_\alpha
  \qquad
  \text{almost everywhere in }\Omega.
  \label{eq:continuum-rearranged-source-bound}
\end{equation}

Let $\varepsilon>0$ and choose $q_\varepsilon\sim G$ such that
\[
  \int_\Omega
  wq_\varepsilon\dx
  \leq
  \mathcal R(w)+\varepsilon.
\]
Using \eqref{eq:continuum-minimum-representation} for $v$ gives
\begin{align*}
  \mathcal R(v)-\mathcal R(w)
  &\leq
  \int_\Omega
  vq_\varepsilon\dx
  -
  \int_\Omega
  wq_\varepsilon\dx
  +
  \varepsilon
  \\
  &\leq
  M_\alpha
  \Norm{v-w}_{L^1(\Omega)}
  +
  \varepsilon.
\end{align*}
Letting $\varepsilon\to0$ and interchanging $v$ and $w$ proves
\eqref{eq:continuum-rearrangement-Lipschitz}.

Finally, $v\mapsto-\mathcal R(v)$ is convex, while
\[
  v
  \longmapsto
  \frac12
  \int_\Omega
  \abs{\nabla v}^2\dx
\]
is strictly convex on $\hoz(\Omega)$. Their sum $\mathcal J$ is
therefore strictly convex.
\end{proof}

\begin{lemma}[Absence of non-maximal plateaux]
\label{lem:continuum-no-secondary-plateaux}
Let $u$ be the continuum Grad--Mercier solution and
$M:=\max_{\overline\Omega}u$. Then
\begin{equation}
  \abs{
    \qc{
      x\in\Omega:
      u(x)=t
    }
  }
  =
  0
  \qquad
  \text{for every }t<M.
  \label{eq:no-lower-positive-measure-levels}
\end{equation}
\end{lemma}

\begin{proof}
Let
\[
  E_t
  :=
  \qc{
    x\in\Omega:
    u(x)=t
  }.
\]
Suppose that $\abs{E_t}>0$. By the Stampacchia level-set property
\cite[Theorem~6.19]{Leoni2009},
\begin{equation}
  \nabla u
  =
  0
  \qquad
  \text{almost everywhere on }E_t.
  \label{eq:gradient-zero-on-level-set}
\end{equation}
Since $u\in \cont{1,1}(\overline\Omega)$, each first derivative
$\partial_j u$ belongs to $W^{1,\infty}(\Omega)$. Moreover,
\eqref{eq:gradient-zero-on-level-set} gives
$\partial_j u=0$ almost everywhere on $E_t$. Applying the same
level-set property to $\partial_j u$ at level zero yields
\[
  \nabla(\partial_j u)
  =
  0
  \qquad
  \text{almost everywhere on }E_t.
\]
Since this holds for every $j$, we obtain
\[
  D^2u
  =
  0
  \qquad
  \text{almost everywhere on }E_t,
\]
and hence $\Delta u=0$ almost everywhere on $E_t$.

On the other hand, $t<M$ implies
\[
  \abs{u\geq t}
  >
  \alpha.
\]
Assumption~\ref{ass:rank-law} therefore gives
\[
  -\Delta u
  =
  g\bigl(
    \abs{u\geq t}
  \bigr)
  >
  0
  \qquad
  \text{almost everywhere on }E_t,
\]
which is a contradiction.
\end{proof}

\begin{lemma}[Continuum variational characterisation]
\label{lem:continuum-variational-characterisation}
Let $u$ be the continuum Grad--Mercier solution and define
\begin{equation}
  q_u(x)
  :=
  G\bigl(
    \abs{u\geq u(x)}
  \bigr).
  \label{eq:continuum-canonical-source}
\end{equation}
Then
\begin{equation}
  \mathcal R(v)-\mathcal R(u)
  \leq
  \int_\Omega
  q_u(v-u)\dx
  \qquad
  \text{for every }v\in L^1(\Omega),
  \label{eq:continuum-supergradient-inequality}
\end{equation}
and $u$ is the unique minimiser of $\mathcal J$.
\end{lemma}

\begin{proof}
Let
\[
  D_u
  :=
  \qc{
    x\in\Omega:
    u(x)=\max_{\overline\Omega}u
  }.
\]
By Theorem~\ref{thm:continuum-well-posedness},
\[
  \abs{D_u}
  =
  \alpha.
\]
Moreover, Lemma~\ref{lem:continuum-no-secondary-plateaux} shows that
every level set below the maximum has zero measure.

Define the non-strict rank map
\[
  r_u(x)
  :=
  \abs{
    \qc{
      y\in\Omega:
      u(y)\geq u(x)
    }
  }.
\]
On the maximum plateau,
\begin{equation}
  r_u(x)
  =
  \alpha
  \qquad
  \text{for every }x\in D_u.
  \label{eq:continuum-plateau-rank}
\end{equation}
Away from $D_u$, the absence of positive-measure level sets implies
that the rank variable is uniformly distributed over
$(\alpha,\abs{\Omega})$. More precisely, writing $\mathcal L^d$ for
$d$-dimensional Lebesgue measure, $\mu|_A$ for the restriction of a
measure $\mu$ to $A$, and $T_\#\mu$ for its pushforward under a map
$T$, we have
\begin{equation}
  (r_u)_\#
  \qp{
    \mathcal L^2|_{\Omega\setminus D_u}
  }
  =
  \mathcal L^1|_{(\alpha,\abs{\Omega})}.
  \label{eq:continuum-rank-pushforward}
\end{equation}
Indeed, let $s\in(\alpha,\abs{\Omega})$. Since the level sets of $u$
below its maximum have zero measure, the distribution function
\[
  t
  \longmapsto
  \abs{u\geq t}
\]
is continuous below the maximum. Hence there exists $t_s$ such that
\[
  \abs{u\geq t_s}
  =
  s.
\]
Up to a null set,
\[
  \qc{
    x\in\Omega\setminus D_u:
    r_u(x)\leq s
  }
  =
  \qc{
    x\in\Omega:
    u(x)\geq t_s
  }
  \setminus D_u,
\]
and therefore
\[
  \abs{
    \qc{
      x\in\Omega\setminus D_u:
      r_u(x)\leq s
    }
  }
  =
  s-\alpha.
\]
This is precisely the distribution function of Lebesgue measure on
$(\alpha,\abs{\Omega})$, and proves
\eqref{eq:continuum-rank-pushforward}.

Set
\begin{equation}
  q_u(x)
  :=
  G\qp{r_u(x)}.
\end{equation}
Since $G=0$ on $[0,\alpha]$, the plateau $D_u$ contributes mass
$\alpha$ at the value zero. Together with
\eqref{eq:continuum-rank-pushforward}, this shows that $q_u$ and $G$
have the same distribution of values. Thus
\begin{equation}
  q_u
  \sim
  G.
  \label{eq:continuum-source-equimeasurable}
\end{equation}

The functions $u$ and $q_u$ are oppositely ordered. Indeed, if
$u(x)>u(y)$, then
\[
  r_u(x)
  \leq
  r_u(y),
\]
and, since $G$ is nondecreasing,
\[
  q_u(x)
  \leq
  q_u(y).
\]
The opposite-order case of the Hardy--Littlewood rearrangement
inequality therefore gives
\begin{equation}
  \mathcal R(u)
  =
  \int_\Omega
  uq_u\dx.
  \label{eq:continuum-source-attains-minimum}
\end{equation}
On the other hand, since $q_u\sim G$, the minimum representation
\eqref{eq:continuum-minimum-representation} gives
\begin{equation}
  \mathcal R(v)
  \leq
  \int_\Omega
  vq_u\dx
  \qquad
  \text{for every }v\in L^1(\Omega).
  \label{eq:continuum-fixed-source-upper-bound}
\end{equation}
Subtracting \eqref{eq:continuum-source-attains-minimum} from
\eqref{eq:continuum-fixed-source-upper-bound} yields
\begin{equation}
  \mathcal R(v)-\mathcal R(u)
  \leq
  \int_\Omega
  q_u(v-u)\dx,
\end{equation}
so $q_u$ is a supergradient of $\mathcal R$ at $u$.

Finally, $r_u(x)\geq\alpha$ throughout $\Omega$, and hence
\[
  q_u(x)
  =
  G\bigl(r_u(x)\bigr)
  =
  g\bigl(r_u(x)\bigr).
\]
The Grad--Mercier equation therefore gives
\[
  \int_\Omega
  \nabla u\cdot\nabla(v-u)\dx
  =
  \int_\Omega
  q_u(v-u)\dx
  \qquad
  \text{for every }v\in \hoz(\Omega).
\]
Using \eqref{eq:continuum-supergradient-inequality} and expanding the
quadratic term, we obtain
\begin{align*}
  \mathcal J(v)-\mathcal J(u)
  &=
  \int_\Omega
  \nabla u\cdot\nabla(v-u)\dx
  +
  \frac12
  \Norm{
    \nabla(v-u)
  }_{L^2(\Omega)}^2
  -
  \bigl(
    \mathcal R(v)-\mathcal R(u)
  \bigr)
  \\
  &\geq
  \frac12
  \Norm{
    \nabla(v-u)
  }_{L^2(\Omega)}^2.
\end{align*}
Thus $u$ minimises $\mathcal J$. The inequality is strict whenever
$v\neq u$, so $u$ is the unique minimiser.
\end{proof}

\subsection{Proof of the main theorem}

We first establish consistency of the discrete rearrangement.

\begin{lemma}[Consistency of the weighted nodal rearrangement]
\label{lem:discrete-rearrangement-consistency}
Let $h\to0$, and suppose that $v_h\in\Vh$ satisfies
\begin{equation}
  \mathcal E_hv_h
  \longrightarrow
  v
  \qquad
  \text{uniformly on }\overline\Omega.
  \label{eq:uniform-function-convergence}
\end{equation}
Then
\begin{equation}
  \mathcal R_h(v_h)
  \longrightarrow
  \mathcal R(v)
  \qquad
  \text{as }h\to0.
  \label{eq:rearrangement-functional-convergence}
\end{equation}
\end{lemma}

\begin{proof}
Set
\[
  m_h
  :=
  \abs{\Omega_h},
  \qquad
  m
  :=
  \abs{\Omega},
\]
and define the distributions of the discrete and limiting functions by
\[
  \nu_h
  :=
  (\mathcal E_hv_h)_\#\muh,
  \qquad
  \nu
  :=
  v_\#
  \qp{
    \mathcal L^2|_\Omega
  }.
\]
We first show that
\begin{equation}
  \nu_h
  \rightharpoonup
  \nu
  \qquad
  \text{weakly as finite measures on }\mathbb R.
  \label{eq:discrete-value-distribution-convergence}
\end{equation}
Indeed, let $\eta\in \cont0(\mathbb R)$. The functions
$\mathcal E_hv_h$ and $v$ are uniformly bounded for sufficiently small
$h$, so only the restriction of $\eta$ to a fixed compact interval is
relevant. By \eqref{eq:uniform-function-convergence},
\[
  \eta(\mathcal E_hv_h)
  \longrightarrow
  \eta(v)
  \qquad
  \text{uniformly on }\overline\Omega.
\]
Consequently, Lemma~\ref{lem:nodal-measure-convergence} gives
\begin{align*}
  \int_{\mathbb R}
  \eta(t)\d\nu_h(t)
  &=
  \int_{\overline\Omega}
  \eta(\mathcal E_hv_h)
  \d\muh
  \\
  &\longrightarrow
  \int_\Omega
  \eta(v)\dx
  =
  \int_{\mathbb R}
  \eta(t)\d\nu(t),
\end{align*}
which proves \eqref{eq:discrete-value-distribution-convergence}.

Choose $B>0$ such that
\[
  \abs{\mathcal E_hv_h}
  \leq
  B,
  \qquad
  \abs v
  \leq
  B
\]
for all sufficiently small $h$, and set
\[
  c_0
  :=
  -B-1.
\]
Since $\Omega_h\subset\Omega$, we have $m_h\leq m$. Define the augmented
measure
\begin{equation}
  \widehat\nu_h
  :=
  \nu_h
  +
  (m-m_h)\delta_{c_0}.
  \label{eq:augmented-pushforward-measure}
\end{equation}
The measures $\widehat\nu_h$ and $\nu$ have the same total mass $m$,
are supported in a common compact interval, and satisfy
\[
  \widehat\nu_h
  \rightharpoonup
  \nu.
\]
Weak convergence on a common compact support implies convergence in the
one-dimensional Wasserstein distance $W_1$. Equivalently, the decreasing
quantiles converge in $L^1(0,m)$.

Let $\widehat v_h^\ast$ denote the decreasing quantile of
$\widehat\nu_h$. Since $c_0$ lies strictly below the values of
$\mathcal E_hv_h$, the added mass occupies the final part of the rank
interval, and hence
\[
  \widehat v_h^\ast(r)
  =
  \begin{cases}
    v_h^\ast(r),
      &0<r<m_h,
    \\
    c_0,
      &m_h<r<m.
  \end{cases}
\]
Therefore
\[
  \int_0^m
  \abs{
    \widehat v_h^\ast(r)-v^\ast(r)
  }
  \d r
  \longrightarrow0,
\]
and in particular
\begin{equation}
  \int_0^{m_h}
  \abs{
    v_h^\ast(r)-v^\ast(r)
  }
  \d r
  \longrightarrow0.
  \label{eq:augmented-quantile-convergence}
\end{equation}

By the definitions of the discrete and continuum rearrangement
functionals,
\[
  \mathcal R_h(v_h)
  =
  \int_0^{m_h}
  G(r)v_h^\ast(r)\d r,
  \qquad
  \mathcal R(v)
  =
  \int_0^m
  G(r)v^\ast(r)\d r.
\]
Since $0\leq G\leq M_\alpha$ and $\abs{v^\ast}\leq B$,
\begin{align*}
  \abs{
    \mathcal R_h(v_h)-\mathcal R(v)
  }
  &\leq
  M_\alpha
  \int_0^{m_h}
  \abs{
    v_h^\ast(r)-v^\ast(r)
  }
  \d r
  \\
  &\quad
  +
  M_\alpha B
  (m-m_h).
\end{align*}
The first term tends to zero by
\eqref{eq:augmented-quantile-convergence}, while the second tends to
zero because $\abs{\Omega_h}\to\abs{\Omega}$. This proves
\eqref{eq:rearrangement-functional-convergence}.
\end{proof}

\begin{proof}[Proof of Theorem~\ref{thm:main-nonlinear}]
Existence, uniqueness and the inclusion
\[
  u_h\in K_h
\]
follow from Lemmas~\ref{lem:discrete-well-posedness}
and~\ref{lem:discrete-barrier-bound}. For each fixed $h$, convergence of
the rearrangement iteration to $u_h$ follows from
Lemma~\ref{lem:rearrangement-iteration-convergence}.

We now consider the limit $h\to0$. Let $h_k\to0$ be arbitrary. For each
$k$, choose any admissible equilibrium source
\[
  q_{h_k}^\ast
  =
  \qp{
    q_{h_k,i}^\ast
  }_{i\in\mathcal N_{h_k}}
\]
satisfying \eqref{eq:discrete-source-equation}, and define
\[
  s_{h_k}^\ast
  :=
  \sum_{i\in\mathcal N_{h_k}}
  q_{h_k,i}^\ast\varphi_i.
\]
The admissible source bounds imply
\[
  0
  \leq
  s_{h_k}^\ast
  \leq
  M_\alpha
  \qquad
  \text{in }\Omega_{h_k}.
\]
Hence there exist a subsequence, not relabelled, and
$f\in \leb{\infty}(\Omega)$ such that
\begin{equation}
  \widetilde{s_{h_k}^\ast}
  \stackrel{*}{\rightharpoonup}
  f
  \qquad
  \text{in }\leb{\infty}(\Omega)
  \quad\text{as }k\to\infty.
  \label{eq:technical-source-subsequence}
\end{equation}

The discrete equilibrium equation is a mass-lumped Poisson problem.
Lemma~\ref{lem:linear-compactness-consistency} therefore gives a
function $w\in \hoz(\Omega)$ such that
\begin{align}
  \mathcal E_{h_k}u_{h_k}
  &\longrightarrow
  w
  &&\text{strongly in }\hoz(\Omega),
  \label{eq:technical-subsequence-H1}
  \\
  \mathcal E_{h_k}u_{h_k}
  &\longrightarrow
  w
  &&\text{strongly in }\leb{\infty}(\Omega),
  \label{eq:technical-subsequence-Linf}
\end{align}
and
\begin{equation}
  \int_\Omega
  \nabla w\cdot\nabla\phi\dx
  =
  \int_\Omega
  f\phi\dx
  \qquad
  \text{for every }\phi\in \hoz(\Omega).
  \label{eq:technical-limit-Poisson}
\end{equation}

Let $\phi\in \cont\infty_c(\Omega)$ and set
\[
  \phi_{h_k}
  :=
  I_{h_k}\phi.
\]
For sufficiently large $k$, the support of $\phi$ lies inside
$\Omega_{h_k}$. Since $q_{h_k}^\ast$ is an admissible supergradient at
$u_{h_k}$,
\begin{align}
  \mathcal R_{h_k}(\phi_{h_k})
  -
  \mathcal R_{h_k}(u_{h_k})
  &\leq
  \sum_{i\in\mathcal N_{h_k}}
  \omega_iq_{h_k,i}^\ast
  \bigl(
    \phi_{h_k}(x_i)-u_{h_k}(x_i)
  \bigr)
  \nonumber\\
  &=
  \sum_{i\in\mathcal N_{h_k}}
  \omega_iq_{h_k,i}^\ast
  \phi_{h_k}(x_i)
  -
  a_{h_k}(u_{h_k},u_{h_k}).
  \label{eq:technical-pass-supergradient}
\end{align}

Lemma~\ref{lem:mass-lumping-error} gives
\[
  \sum_{i\in\mathcal N_{h_k}}
  \omega_iq_{h_k,i}^\ast
  \phi_{h_k}(x_i)
  -
  \int_{\Omega_{h_k}}
  s_{h_k}^\ast\phi_{h_k}\dx
  \longrightarrow0.
\]
Moreover,
\[
  \mathcal E_{h_k}\phi_{h_k}
  \longrightarrow
  \phi
  \qquad
  \text{uniformly on }\overline\Omega.
\]
Together with
\eqref{eq:technical-source-subsequence}, this gives
\begin{equation}
  \int_{\Omega_{h_k}}
  s_{h_k}^\ast\phi_{h_k}\dx
  \longrightarrow
  \int_\Omega
  f\phi\dx.
  \label{eq:technical-source-test-limit}
\end{equation}

The strong $\hoz(\Omega)$ convergence in
\eqref{eq:technical-subsequence-H1} gives
\[
  a_{h_k}(u_{h_k},u_{h_k})
  =
  \Norm{
    \nabla\mathcal E_{h_k}u_{h_k}
  }_{L^2(\Omega)}^2
  \longrightarrow
  \Norm{\nabla w}_{L^2(\Omega)}^2.
\]
Testing \eqref{eq:technical-limit-Poisson} with $\phi=w$ then yields
\begin{equation}
  a_{h_k}(u_{h_k},u_{h_k})
  \longrightarrow
  \int_\Omega
  fw\dx.
  \label{eq:technical-discrete-energy-limit}
\end{equation}

Lemma~\ref{lem:discrete-rearrangement-consistency}, applied to
\eqref{eq:technical-subsequence-Linf} and to the interpolants
$\phi_{h_k}$, gives
\begin{align}
  \mathcal R_{h_k}(u_{h_k})
  &\longrightarrow
  \mathcal R(w),
  \label{eq:technical-solution-rearrangement-limit}
  \\
  \mathcal R_{h_k}(\phi_{h_k})
  &\longrightarrow
  \mathcal R(\phi).
  \label{eq:technical-test-rearrangement-limit}
\end{align}
Passing to the limit in
\eqref{eq:technical-pass-supergradient} therefore yields
\begin{equation}
  \mathcal R(\phi)-\mathcal R(w)
  \leq
  \int_\Omega
  f(\phi-w)\dx
  \qquad
  \text{for every }\phi\in \cont\infty_c(\Omega).
  \label{eq:technical-limit-supergradient}
\end{equation}

Since $\cont\infty_c(\Omega)$ is dense in $L^1(\Omega)$,
$f\in \leb{\infty}(\Omega)$, and $\mathcal R$ is Lipschitz continuous on
$L^1(\Omega)$, the inequality extends to
\begin{equation}
  \mathcal R(\phi)-\mathcal R(w)
  \leq
  \int_\Omega
  f(\phi-w)\dx
  \qquad
  \text{for every }\phi\in L^1(\Omega).
  \label{eq:technical-limit-supergradient-L1}
\end{equation}

For $\phi\in \hoz(\Omega)$,
\eqref{eq:technical-limit-Poisson} gives
\[
  \int_\Omega
  \nabla w\cdot\nabla(\phi-w)\dx
  =
  \int_\Omega
  f(\phi-w)\dx.
\]
Combining this identity with
\eqref{eq:technical-limit-supergradient-L1} yields
\[
  \int_\Omega
  \nabla w\cdot\nabla(\phi-w)\dx
  \geq
  \mathcal R(\phi)-\mathcal R(w)
  \qquad
  \text{for every }\phi\in \hoz(\Omega).
\]
Thus $w$ minimises the continuum energy $\mathcal J$. By
Lemma~\ref{lem:continuum-variational-characterisation}, the unique
minimiser is the continuum Grad--Mercier solution, and therefore
\begin{equation}
  w
  =
  u.
  \label{eq:technical-subsequence-identification}
\end{equation}

It remains to identify the weak-star limit of the reconstructed sources.
Substituting $w=u$ into \eqref{eq:technical-limit-Poisson} gives
\begin{equation}
  \int_\Omega
  \nabla u\cdot\nabla\phi\dx
  =
  \int_\Omega
  f\phi\dx
  \qquad
  \text{for every }\phi\in \hoz(\Omega).
  \label{eq:technical-limit-source-equation}
\end{equation}
On the other hand, the weak formulation of the continuum
Grad--Mercier equation gives
\begin{equation}
  \int_\Omega
  \nabla u\cdot\nabla\phi\dx
  =
  \int_\Omega
  g\bigl(
    \abs{u\geq u(\cdot)}
  \bigr)
  \phi\dx
  \qquad
  \text{for every }\phi\in \hoz(\Omega).
  \label{eq:technical-continuum-source-equation}
\end{equation}
Subtracting
\eqref{eq:technical-continuum-source-equation} from
\eqref{eq:technical-limit-source-equation} gives
\[
  \int_\Omega
  \qb{
    f
    -
    g\qp{
      \abs{u\geq u(\cdot)}
    }
  }
  \phi\dx
  =
  0
  \qquad
  \text{for every }\phi\in \cont\infty_c(\Omega).
\]
Therefore
\begin{equation}
  f
  =
  g\qp{
    \abs{u\geq u(\cdot)}
  }
  \qquad
  \text{almost everywhere in }\Omega.
  \label{eq:technical-source-identification}
\end{equation}

The sequence $h_k\to0$ was arbitrary, and every convergent subsequence
has the same limit $u$. The subsequence principle therefore gives
\begin{align*}
  \mathcal E_hu_h
  &\longrightarrow
  u
  &&\text{strongly in }\hoz(\Omega),
  \\
  \mathcal E_hu_h
  &\longrightarrow
  u
  &&\text{strongly in }\leb{\infty}(\Omega)
\end{align*}
as $h\to0$.

The admissible equilibrium sources were also chosen arbitrarily. Every
weak-star accumulation point of their reconstructions is identified by
\eqref{eq:technical-source-identification}. Consequently, for every
choice of admissible equilibrium sources $q_h^\ast$,
\[
  \widetilde{s_h^\ast}
  \stackrel{*}{\rightharpoonup}
  g\qp{
    \abs{u\geq u(\cdot)}
  }
  \qquad
  \text{in }\leb{\infty}(\Omega)
  \quad\text{as }h\to0.
\]

Finally, the maximum plateau has measure $\alpha$ by
Theorem~\ref{thm:continuum-well-posedness}, and the absence of
positive-measure level sets below the maximum follows from
Lemma~\ref{lem:continuum-no-secondary-plateaux}.
\end{proof}

\section*{Acknowledgments}

TP is supported by the EPSRC programme grant Mathematics of Radiation
Transport (MaThRad) EP/W026899/2. All of this support is gratefully
acknowledged.

\printbibliography

\end{document}